\documentclass[10pt, reqno]{amsart}
\usepackage{amsaddr}
\usepackage{graphicx,color}
\usepackage{float}
\usepackage{hyperref}
 \usepackage{mathtools}
 \usepackage{amsmath,amssymb}
 \usepackage{amsthm}
 \usepackage{empheq}
 \DeclareMathOperator\supp{supp}
\usepackage[a4paper,
bindingoffset=0.2in,
left=1in,
right=1in,
top=1in,
bottom=1in,
footskip=.25in]{geometry}
\usepackage{caption}
\usepackage{subcaption}
\usepackage{concmath}
\newtheorem{thm}{Theorem}[section]

\newtheorem{lem}[thm]{Lemma}
\newtheorem{prop}[thm]{Proposition}
\theoremstyle{definition}
\newtheorem{defn}[thm]{Definition}
\newtheorem{rem}[thm]{Remark}
\numberwithin{equation}{section}

\newcommand{\Real}{\mathbb R}

\begin{document}

\title[Schrödinger equation on the half line]{The interior-boundary Strichartz estimate for the Schrödinger equation on the half line revisited}%
\author{B\.{I}lge Köksal and Türker Özsarı}
\address{Department of Mathematics, Bilkent University, Ankara, 06800 Turkey}%
\email{turker.ozsari@bilkent.edu.tr}%

\thanks{TÖ's research is partially supported by the Science Academy's Young Scientist Award Program (BAGEP 2020)}%

\begin{abstract} It was shown by the second author in \cite{OY19} for the biharmonic Schrödinger equation and most recently by Himonas and Mantzavinos \cite{Him20} for 2D Schrödinger equation that Fokas method based formulas are capable of defining weak solutions of associated nonlinear initial boundary value problems (ibvps) below the Banach algebra threshold. In view of these results, we revisit the theory of interior-boundary Strichartz estimates for the Schrödinger equation posed on the right half line, considering both Dirichlet and Neumann cases.  Finally, we apply these estimates to obtain low regularity solutions for the nonlinear Schrödinger equation (NLS) with Neumann boundary condition and a coupled system of NLS equations defined on the half line with Dirichlet/Neumann boundary conditions.
\end{abstract}

\maketitle
\section{Introduction}

\setlength{\parskip}{\baselineskip}%
\setlength{\parindent}{0pt}%

Well-posedness of the initial-boundary value problem (ibvp) for the nonlinear Schrödinger equation (NLS)  on various domains with inhomogeneous boundary data has been studied in several papers, see e.g.; \cite{Aud19}, \cite{Bona18}, \cite{Esq19}, \cite{FHM17}, \cite{Hay21}, \cite{Hay21b}, \cite{Holmer}, \cite{Ozs12}, \cite{Ozs18},  \cite{Ozs15}, \cite{Ozs11}, \cite{Str11}. Here we are concerned with the particular case where the spatial domain is the right half-line. Consider first the ibvp for the NLS given by
\begin{align}
&y_{t}+Py=f(y), \quad (x,t)\in\mathbb{R}_+\times (0,T),\label{maineq}\\
&y(x,0)=y_0(x),\label{init}\\
&[By](t)=g(t)\label{bdry},
\end{align} where $P=-i\partial_x^2$ (formal Schrödinger operator), $f$ is the complex valued power type function defined by $f(y)=\lambda |y|^py$ with $\lambda\in \mathbb{C}$, $p>0$, and $B$ is a trace operator given by $B=\gamma_0$ (Dirichlet trace) or $B=\gamma_1$ (Neumann trace) but more general boundary conditions could also be considered.  Solutions of the nonlinear problem \eqref{maineq}-\eqref{bdry} can be obtained by applying a fixed point theorem to the linear solution operator, which is constructed based on a formula for solutions of an associated linear ibvp. The latter problem is usually studied via a decompose-and-reunify approach. In this approach, one decomposes the problem into three subproblems (i) a homogeneous Cauchy problem with no interior source, (ii) a nonhomogeneous Cauchy problem with zero initial value, and (iii) an ibvp with zero initial and interior data. Spatial regularity of both types of Cauchy problems for NLS are widely studied in the literature, see for instance Cazenave's book \cite{Caz03} - a classical reference on this topic. However, much less effort was given for the temporal regularity of these Cauchy problems as well as the spatial regularity of the ibvp with an inhomogeneous boundary datum. We are aware of some approaches regarding the treatment of the ibvp for the half line problem.  Holmer \cite{Holmer} studied this problem by constructing a boundary forcing operator based on the Riemann-Liouville fractional integral, an approach that was previously applied to the ibvp for the Korteweg-de Vries (KdV) equation posed on the half line \cite{col02}. Some recent papers studied the same ibvp by analyzing the solution formula constructed with one of the traditional integral transforms. For instance, Bona-Sun-Zhang \cite{Bona18} used the Laplace transform in temporal variable, and Esquivel-Hayashi-Kaikina \cite{Esq19} used the Fourier sine transform. Finally, Fokas-Himonas-Mantzavinos \cite{FHM17} introduced an approach  utilizing the integral representation formula obtained through the unified transform method of Fokas \cite{Fbook}.

In \cite{FHM17}, authors treated \eqref{maineq}-\eqref{bdry} at the high regularity level (i.e., in $H_x^s(\mathbb{R}_+)$ with $s>1/2$) and obtained estimates in the $L_t^\infty(0,T;H_x^s(\mathbb{R}_+))$ norm.  In this setting, $H_x^s(\mathbb{R}_+)$ is a Banach algebra, i.e., $$|\phi\psi|_{H_x^s(\mathbb{R}_+)}\lesssim |\phi|_{H_x^s(\mathbb{R}_+)}|\psi|_{H_x^s(\mathbb{R}_+)} \text{ for } \phi,\psi\in {H_x^s(\mathbb{R}_+)},$$ and therefore handling the nonlinearities via contraction is relatively easier. In the low regularity setting $s\le \frac{1}{2}$, $H_x^s(\mathbb{R}_+)$ looses its algebra structure, and estimates in the $L_t^\infty(0,T;H_x^s(\mathbb{R}_+))$ norm are not good enough to perform the associated nonlinear analysis.  The classical method in the theory of nonlinear dispersive PDEs for dealing with this difficulty is to prove Strichartz type estimates in mixed norm function spaces $L_t^\lambda(0,T;H_x^{{s,r}}(\mathbb{R}_+))$, where $(\lambda,r)$ satisfies a special {admissibility} condition intrinsic to the underlying evolution operator. Holmer \cite{Holmer} and Bona, et al. \cite{Bona18} gave proofs of such estimates for the Dirichlet problem in the low regularity setting by analyzing the representation formulas obtained through Riemann–Liouville fractional integral and Laplace transform, respectively.

Strichartz estimates, first noted in \cite{Strich} within the framework of the Fourier restriction problem, are a group of inequalities for linear dispersive PDEs that allow us to bound the size and decay of solutions in mixed norm Lebesgue-Sobolev spaces.  These estimates are mostly established for Cauchy type problems where the spatial domain is the whole Euclidean space.  The results on other geometries are rather limited, and not as strong. In these results, either the given geometry has no boundary (e.g., a boundaryless manifold) or else the boundary conditions are set to zero.  Therefore, the estimates are still given with respect to only initial and interior data.  On the other hand, Strichartz estimates for inhomogeneous ibvps are rare, and there are relatively much fewer work in this direction.

Strichartz estimates generally rely on local-in-time dispersive estimates. In second author's work \cite{OY19} on the biharmonic nonlinear Schrödinger equation (BNLS) with Dirichlet-Neumann boundary conditions, it was shown that these estimates can also be proven through the analysis of the representation formula obtained via the Fokas method. Most recently, Himonas and Mantzavinos \cite{Him20} made use of the same idea for obtaining the low regularity solutions of NLS posed on the half plane with Dirichlet boundary condition.  In view of these two papers, one can also expect that recent UTM based $L_t^\infty H_x^s(\mathbb{R}_+)$ estimates of Fokas, et al. \cite{FHM17} should extend to UTM based $L_t^\lambda H_x^{s,r}(\mathbb{R}_+)$ estimates in the one dimensional setting. The goal of this paper is to revisit the one dimensional theory, prove Strichartz estimates and apply them to NLS and a system of coupled NLS equations defined on the half line with Dirichlet or Neumann boundary conditions. The Neumann case is a topic which was not treated also with other methods; \cite{Bona18} and \cite{Holmer} only considered the Dirichlet case.  The second author's paper \cite{Batal16} and later Himonas, et. al. \cite{Him19} treated the Neumann problem with Laplace transform and UTM based formulas, respectively, however both of these papers handle only the high regularity solutions (i.e., $s>1/2$).

\subsection*{Outline of the paper}

The paper is organized as follows. In Section \ref{decomSec}, we review the decompose and reunify algorithm for treatment of the linear ibvp, in particular give the Fokas method based representation formulas for Dirichlet and Neumann problems. In Section \ref{quickreview}, we look over the Strichartz estimates for Cauchy problems and review known time estimates. In Section \ref{statebc} we state boundary Strichartz estimates for Dirichlet and Neumann problems. Sections \ref{pf1}-\ref{pf2} present the proofs of these estimates. Finally, in Section \ref{Apps}, we apply the linear estimates to prove local wellposedness of NLS and coupled system of NLS equations.

\section{Decompose-and-reunify}\label{decomSec} In this section, we review the decompose-and-reunify algorithm to study \eqref{maineq}-\eqref{bdry}.
Abusing the notation, we first write the associated linear nonhomogeneous problem:
\begin{subequations}\label{lineq}
\begin{empheq}{align}
&y_{t}+Py=f(x,t), \quad (x,t)\in\mathbb{R}_+\times (0,T),\label{maineqlin}\\
&y(x,0)=y_0(x),\label{initlin}\\
&[By](t)=g(t)\label{bdrylin},
\end{empheq}
\end{subequations}
where $f:\mathbb{R_+}\times (0,T)\rightarrow \mathbb{C}$.
We fix spatial bounded extension operators $(\cdot)^*$, say from a Sobolev space defined on $\mathbb{R}_+$, into a Sobolev space defined on $\mathbb{R}$ and consider (i) a homogeneous Cauchy problem with nonzero initial datum, (ii) a nonhomogeneous Cauchy problem with zero initial datum and (iii) an ibvp with zero initial and interior data:
\begin{subequations}\label{v}
\begin{empheq}{align}
&v_{t}+Pv=0, \quad (x,t)\in\mathbb{R}\times (0,T),\label{maineqlinHC}\\
&v(x,0)=y_0^*(x),\label{initlinHC}
\end{empheq}
\end{subequations}
\begin{subequations}\label{z}
\begin{empheq}{align}
&z_{t}+Pv=f^*(x,t), \quad (x,t)\in\mathbb{R}\times (0,T),\label{maineqlinNHC}\\
&z(x,0)=0,\label{initlinNHC}
\end{empheq}
\end{subequations}
\begin{subequations}\label{u}
\begin{empheq}{align}
&u_{t}+Pu=0, \quad (x,t)\in\mathbb{R}_+\times (0,T'),\label{maineqlinH}\\
&u(x,0)=0,\label{initlinH}\\
&[Bu](t)=h(t)\label{bdrylinH}.
\end{empheq}
\end{subequations}
In the above equations, $y_0^*$ and $f^*$ are spatial extensions of initial and interior data, $$h(t)\equiv g(t)-[Bv](t)-[Bz](t)$$ in which for convenience we extend the RHS beyond the given time interval $(0,T)$ such that it is zero for $t>T'$ for some $T'>T$. The condition, $T'>T$ is convenient for a smooth transition to zero so that the given regularity level of $h$ on $(0,T)$ is preserved on the extended interval. Now, the solution of \eqref{maineqlin}-\eqref{bdrylin} can be defined via reunification as follows:
\begin{equation}\label{reunify}
  y=v|_{\mathbb{R}_+\times (0,T)}+z|_{\mathbb{R}_+\times (0,T)}+u|_{(0,T)}.
\end{equation}
Note that we subtract the boundary traces associated with the two Cauchy problems when we define $h$, therefore one needs to know existence of these traces. A representation formula for the solution of the homogeneous Cauchy problem \eqref{maineqlinHC}-\eqref{initlinHC} is formally given via Fourier transform as
\begin{equation}\label{vform}
\begin{split}
  v(x,t)&=\frac{1}{2\pi}\int_{-\infty}^{\infty}e^{ikx-ik^2t}\widehat{y_0^*}(k)dk.
\end{split}
  \end{equation} We will use the notation $v(x,t)=e^{-tP}y_0^*(x)$ to denote the solution of the homogeneous Cauchy problem.  The above formula is well defined for $y_0^*\in \mathcal{S},$ and one has
$            \hat{v}(k,t) = e^{-ik^2t}\widehat{y_0^*}(k)
          $ for $y_0^*\in \mathcal{S}$, which implies the conservation property $|e^{-tP}y_0^*|_{H^s(\mathbb{R})} = |y_0^*|_{H^s(\mathbb{R})}$ for $y_0^*\in \mathcal{S}$. Since $\mathcal{S}$ is dense in $H^s(\mathbb{R})$, $e^{-tP}$ easily extends to a group of isometries on $H^s(\mathbb{R})$ (still denoted same).  Moreover, we have a Duhamel formulation for the solution of \eqref{maineqlinNHC}-\eqref{initlinNHC}:
\begin{equation}\label{zform}
  z(x,t)=\int_0^te^{-(t-t')P}{f^*}(x,t')dt'.
\end{equation}
The solution of the ibvp \eqref{maineqlinH}-\eqref{bdrylinH} with $B=\gamma_0$ (Dirichlet b.c.), and zero initial and interior data is given by the Fokas method as a complex integral \cite{Fbook}:
\begin{equation}\label{uform_dir}
  u(x,t)=\frac{1}{\pi}\int_{\partial D^+}e^{ikx-ik^2t}k\tilde{h}(k^2,T')dk,
\end{equation} where $D^+=\{k\in \mathbb{C}_+\,|\,\Re(ik^2)<0\},$ its boundary $\partial D^+$ is positively oriented (see Figure \ref{Dplus}), and $\tilde{h}(k^2,T')=\int_0^{T'}e^{ik^2s}h(s)ds.$
The solution in the case $B=\gamma_1$ (Neumann b.c.) takes the form
\begin{equation}\label{uform_neu}
  u(x,t)=-\frac{i}{\pi}\int_{\partial D^+}e^{ikx-ik^2t}\tilde{h}(k^2,T')dk.
\end{equation}
We will use the notation $\mathcal{T}_B(t)h$ to denote the solution of the ibvp, namely to denote the right hand side of \eqref{uform_dir} or \eqref{uform_neu} (depending on $B$) obtained through the Fokas method.
Therefore, we can rewrite \eqref{reunify} as follows:
\begin{equation}
\begin{split}\label{reunify2}
  y(t) =\left.e^{-tP}y_0^*\right|_{\mathbb{R}_+\times (0,T)}+\left.\int_0^te^{-(t-t')P}{f^*}(t')dt'\right|_{\mathbb{R}_+\times (0,T)}+\left.\mathcal{T}_B(t)h\right|_{(0,T)}.
\end{split}
\end{equation}
\begin{figure}[h]
  \centering
  \includegraphics[width=4cm]{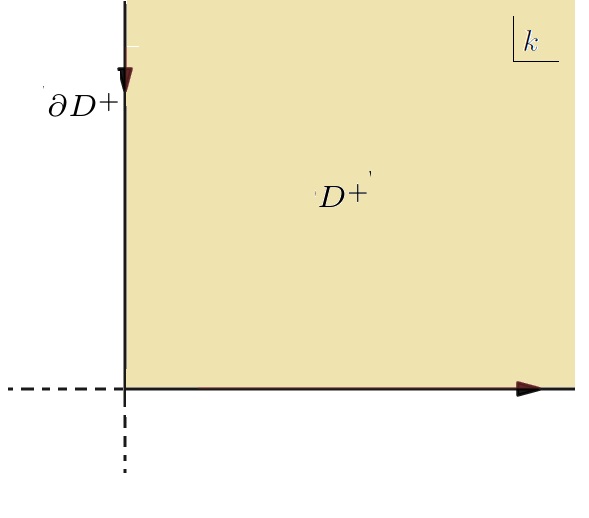}
  \caption{The contour of integration}\label{Dplus}
\end{figure}
\begin{rem}
  We recall that the integral representation formulas \eqref{uform_dir} and \eqref{uform_neu} is first obtained by assuming that $h$ is smooth and has sufficient decay.  However, it is then shown that the same formula makes sense under much weaker regularity conditions imposed on $h$ as shown in this paper.  Therefore, this formula in particular defines weak solutions.
\end{rem}
\section{Linear estimates}\label{StrSec}
\subsection{A quick review}\label{quickreview}
Cauchy problems \eqref{v} and \eqref{z} are well studied in the literature. For instance, regarding the homogeneous Cauchy problem, we have the theorem below in which the Schrödinger admissibility condition
\begin{equation}\label{adm}
  \frac{1}{\lambda}+\frac{1}{2r}=\frac{1}{4},\quad 2\le \lambda,r\le \infty
\end{equation} between indices $\lambda$ and $r$ is important.
\begin{thm}[\cite{Holmer}]\label{Cauchythm}
  Let $s\in \mathbb{R}$, $y_0^*\in H_x^{s}(\mathbb{R})$, $(\lambda,r)$ be Schrödinger admissible. Then $v(t)=e^{-tP}y_0^*$ defines a solution to \eqref{vform} that belongs to $C([0,T];H_x^s(\mathbb{R_+}))\cap C(\mathbb{R}; H_t^{\frac{2s+1}{4}}(0,T))$ such that \begin{itemize}
         \item[(i)] $|v|_{C([0,T];H_x^s(\mathbb{R_+}))}\lesssim |y_0^*|_{H_x^{s}(\mathbb{R})}$,
         \item[(ii)] $\displaystyle\sup_{x\in \mathbb{R}} |v(x)|_{H_t^{\frac{2s+1}{4}}(0,T)}\lesssim \langle T\rangle^\frac{1}{4}|y_0^*|_{H_x^{s}(\mathbb{R})},$
         \item[(iii)] $|v|_{L_t^\lambda(0,T; H_x^{s,r}(\mathbb{R}))}\lesssim |y_0^*|_{H_x^{s}(\mathbb{R})},$ where constants of inequalities depend only on $s$.
       \end{itemize}
\end{thm}
For the nonhomogeneous Cauchy problem, the following theorem is known:
\begin{thm}[\cite{Bona18}, \cite{Holmer}]\label{NonhomCauchythm}
 Let $s\in \mathbb{R}$, $(\lambda,r)$ be Schrödinger admissible, $f^*\in L_t^{\lambda'}H_x^{s,r'}$. Then $z$ given by \eqref{zform} belongs to $C([0,T];H_x^s(\mathbb{R_+}))\cap C(\mathbb{R}; H_t^{\frac{2s+1}{4}}(0,T))$ such that \begin{itemize}
         \item[(i)] $|z|_{C([0,T];H_x^s(\mathbb{R_+}))}\lesssim |f^*|_{L_t^{\lambda'}(0,T;H_x^{s,r'}(\mathbb{R}))}$,
         \item[(ii)] if $-\frac{3}{2}<s<\frac{1}{2}$, then $\displaystyle\sup_{x\in \mathbb{R}} |z(x)|_{H_t^{\frac{2s+1}{4}}(0,T)}\lesssim \langle T\rangle^\frac{1}{4}|f^*|_{L_t^{\lambda'}(0,T;H_x^{s,r'}(\mathbb{R}))},$
         \item[(iii)] if $\frac{1}{2}<s<\frac{5}{2}$ and $\displaystyle\sup_{x\in\mathbb{R}}|f^*(x,\cdot)|_{H_t^\frac{2s-3}{4}}<\infty$, then $$\displaystyle\sup_{x\in \mathbb{R}} |z(x)|_{H_t^{\frac{2s+1}{4}}(0,T)}\lesssim \left(\sup_{x\in\mathbb{R}}|f^*(x,\cdot)|_{H_t^\frac{2s-3}{4}}+|f^*|_{L_t^{\lambda'}(0,T;H_x^{s,r'}(\mathbb{R}))}\right),$$
         \item[(iv)] $|z|_{L_t^\lambda(0,T; H_x^{s,r}(\mathbb{R}))}\lesssim |f^*|_{L_t^{\lambda'}(0,T;H_x^{s,r'}(\mathbb{R}))},$ where constants of inequalities depend only on $s$ except in item (iii), in which it also depends on $T$.
       \end{itemize}
\end{thm}
Regarding the last term in \eqref{reunify2} obtained through the Fokas method, we know the following result for $B=\gamma_0$ (Dirichlet b.c.):
\begin{thm}[\cite{FHM17}]\label{ibvpthm}
  Suppose $s>\frac{1}{2}$ and $h\in H_t^{\frac{2s+1}{4}}(\mathbb{R})$ with $supp\,h \subset [0,T')$ so that it satisfies necessary compatibility conditions. Then, $u$ given by \eqref{uform_dir} belongs to $C([0,T'];H_x^s(\mathbb{R_+}))\cap C(\mathbb{R}_+; H_t^{\frac{2s+1}{4}}(0,T'))$ such that \begin{itemize}
         \item[(i)] $|u|_{C([0,T'];H_x^s(\mathbb{R_+}))}\lesssim |h|_{H_t^{\frac{2s+1}{4}}(\mathbb{R})}$,
         \item[(ii)] $\displaystyle\sup_{x\ge 0} |u(x)|_{H_t^{\frac{2s+1}{4}}(0,T')}\lesssim |h|_{H_t^{\frac{2s+1}{4}}(\mathbb{R})},$
       \end{itemize} where constants of inequalities depend on $s$.
\end{thm}
\begin{rem}
The estimates in Theorem \ref{Cauchythm}-(i), (iii), Theorem \ref{NonhomCauchythm}-(i), (iv), and Theorem \ref{ibvpthm}-(i)  are referred to as space estimates, and the estimates in Theorem \ref{Cauchythm}-(ii), Theorem \ref{NonhomCauchythm}-(ii), (iii), and  Theorem \ref{ibvpthm}-(ii) are time estimates. The indices $\lambda$ and $r$ in the spaces $L^\lambda_t(0,T';H^{s,r}_x(\mathbb{R}_+)$ are not random, and, as noted in Theorem \ref{Cauchythm} and Theorem \ref{NonhomCauchythm}, they must obey an admissibility condition intrinsic to the underlying differential operator.  The space estimates in Theorem \ref{Cauchythm}-(i), Theorem \ref{NonhomCauchythm}-(i), and Theorem \ref{ibvpthm}-(i) are essentially a special case of Strichartz estimates, namely they are $L^\lambda_t(0,T';H^{s,r}_x(\mathbb{R}_+)$ type estimates with $\lambda=\infty$ and $r=2$.
\end{rem}

\subsection{Boundary type Strichartz estimates for Fokas formulas}\label{statebc}
In this section, we state the Strichartz estimates for Fokas method formulas representing solutions of the simplified ibvp.
\begin{thm}[Dirichlet b.c.]\label{mainthm}
	Let $s\ge 0$, $h\in H_t^{\frac{2s+1}{4}}(\mathbb{R})$ with $supp\,h \subset [0,T')$ satisfying necessary compatibility conditions, and $(\lambda,r)$ be Schrödinger admissible. Then, the Fokas method based formula \begin{equation}u(x,t)\equiv \frac{1}{\pi}\int_{\partial D^+}e^{ikx-ik^2t}k\tilde{h}(k^2,T')dk,\end{equation} defines a function $u\in C([0,T'];H_x^s(\mathbb{R}_+))$ that satisfies the Strichartz estimate \begin{equation}\label{StrEst01}
		|u|_{L_t^\lambda(0,T'; H_x^{s,r}(\mathbb{R}_+))}\lesssim |h|_{H_t^{\frac{2s+1}{4}}(\mathbb{R})},
	\end{equation} where the constant of the inequality depends on $s$.
\end{thm}
\begin{thm}[Neumann b.c.]\label{mainthm2}
	Let $s\ge 0$, $h\in H_t^{\frac{2s-1}{4}}(\mathbb{R})$ with $supp\,h \subset [0,T')$ satisfying necessary compatibility conditions, and $(\lambda,r)$ be Schrödinger admissible. Then, the function $u$ defined by the Fokas method based formula \begin{equation}u(x,t)\equiv -\frac{i}{\pi}\int_{\partial D^+}e^{ikx-ik^2t}\tilde{h}(k^2,T')dk\end{equation} satisfies the homogeneous Strichartz estimate \begin{equation}\label{StrEst02}|u|_{L_t^\lambda(0,T'; \dot{H}_x^{s,r}(\mathbb{R}_+))}\lesssim |h|_{\dot{H}_t^{\frac{2s-1}{4}}(\mathbb{R})}.\end{equation} for $s\in \mathbb{N}_0$. If $s\ge 0$, then $u\in C([0,T'];H_x^s(\mathbb{R}_+))$ and it satisfies the inhomogeneous Strichartz estimate:
	\begin{equation}\label{StrEst03}|u|_{L_t^\lambda(0,T'; H_x^{s,r}(\mathbb{R}_+))}\lesssim c_{T'}|h|_{H_t^{\frac{2s-1}{4}}(\mathbb{R})}.\end{equation} In both estimates, the constant of the inequality depends on $s$.
\end{thm}
\begin{rem}
	The Strichartz estimates in Theorem \ref{mainthm2} for the Neumann problem are new to the best of our knowledge.  Note that the constant of the inhomogeneous estimate in Theorem \ref{mainthm2} depends on $T'$ while the estimate in Theorem \ref{mainthm} is independent of $T'$.  However, it is implied by the proof that the dependence of the estimate on $T'$ in Theorem \ref{mainthm2} is nice in the sense that if one studies the corresponding nonlinear problem via contraction it will not cause any issues whatsoever.
\end{rem}
\begin{rem}
Note that inhomogeneous Strichartz estimates in Theorem \ref{mainthm2} cover both the high and low regularity settings.
\end{rem}
\subsection{Proof of Theorem \ref{mainthm}}\label{pf1}

\subsubsection{Splitting the solution of the ibvp}

In order to prove the Strichartz estimates, we first split the ibvp solution $u$ in two parts by using the definition and relevant parametrization of the boundary of $D^+$.  So, we have
\begin{equation}\label{decompu}
	\begin{split}
		u(x,t) & = \frac{1}{\pi}\int_{\partial D^+}e^{ikx-ik^2t}k\tilde{h}(k^2,T')dk\\
		&= -\frac{1}{\pi}\int_\infty^0e^{-kx+ik^2t}k\hat{h}(k^2)dk+\frac{1}{\pi}\int_{0}^\infty e^{ikx-ik^2t}k\hat{h}(-k^2)dk\\
		&\equiv  u_1(x,t)+u_2(x,t).
	\end{split}
\end{equation}

Sometimes, we will write $\mathcal{T}_{B,i}(t)h$ to denote $u_i(\cdot,t)$, $i=1,2$.

Note that in the above decomposition, we in addition used the fact that $\tilde{h}(k^2,T')=\hat{h}(-k^2)$ which follows from the support condition $supp\,h \subset [0,T')$.  This relation is of particular importance for relating the estimates in the next two sections to the Sobolev norm of the boundary input.

\subsubsection{Analysis on the imaginary axis: oscillatory kernel} In this section, we will prove Strichartz estimates for $u_1$. To this end, we set a function $H_1$ which is defined to be the inverse Fourier transform of the function below:
\begin{equation}\label{H1hat}
	\hat{H}_1(k)\equiv \left\{
	\begin{array}{ll}
		\frac{1}{\pi}k\hat{h}(k^2), & \hbox{if $k\ge 0$;} \\
		0, & \hbox{otherwise}.
	\end{array}
	\right.
\end{equation}
Then, upon changing the order of integrals, we can represent $u_1$ as
\begin{equation}\label{rewriteu1}
	\begin{split}
		u_1(x,t)&=-\frac{1}{\pi}\int_\infty^0e^{-kx+ik^2t}k\hat{h}(k^2)dk = \int_{-\infty}^{\infty} e^{-kx+ik^2t}\hat{H}_1(k)dk\\
		&=\lim_{b\rightarrow \infty}\int_{0}^{b}e^{-kx+ik^2t}\int_{-\infty}^\infty e^{-ik\tau}H_1(\tau)d\tau dk\\
		&=\lim_{b\rightarrow \infty}\int_{-\infty}^{\infty}\ell(\tau;x,t,b)H_1(\tau)d\tau,
	\end{split}
\end{equation} where
\begin{equation}\label{kernel}
	\ell(\tau;x,t,b)=\int_{0}^b e^{-kx+ik^2t-ik\tau}dk=\int_{0}^b e^{i\phi(k;\tau,t)}\psi(k;x)dk.
\end{equation} In the oscillatory integral \eqref{kernel}, $\psi(k;x)\equiv e^{-kx}$ is the amplitude function and $$\phi(k;\tau,t)\equiv k^2t-{k}\tau=t\varphi(k)-k\tau$$ is the phase function with $\varphi(k)=k^2$.

\begin{defn}The function $\ell$ in \eqref{rewriteu1} will be referred to as the \emph{kernel of the representation}.
\end{defn}

We first recall the following lemma from the oscillatory integral theory:

\begin{lem}[\cite{Ken91}]\label{Kenlem} Let $I(\tau,t,k)=\int_0^{k}e^{i(t\varphi(k')-k'\tau)}dk'$ with $\varphi(k')={(k')}^2$. Then $$|I(\tau,t,k)|\le c_\varphi |t|^{-1/2}, \tau,t,k\in \mathbb{R}$$ where $c_\varphi>0$ is a constant independent of $\tau,t,k$.
\end{lem}

Now, we can state the following decay estimate for the kernel of the representation:
\begin{lem}\label{vander}The kernel defined by \eqref{kernel} satisfies the following dispersive estimate:
	$$\left|\ell(\tau;x,t,b)\right|\le \frac{c}{\sqrt{|t|}},\quad t\neq 0, \text{ uniformly in }x,b\in\mathbb{R}_+, \tau\in\mathbb{R}.$$
\end{lem}
\begin{proof} We set $\Phi(k;\tau,t)\equiv \int_{0}^ke^{i\phi(k';\tau,t)}dk'$.  Then, $$\ell(\tau;x,t,b)=\int_0^b\left[\frac{d}{dk}\Phi(k;\tau,t)\right]\psi(k;x)dk.$$ Integrating in the RHS, using $\Phi(0;\tau,t)=0$, and estimating $|\ell(\tau;x,t,b)|$, we get
	$$\left|\ell(\tau;x,t,b)\right| \le |\Phi(b;\tau,t)\psi(b;x)|+\int_0^b|\Phi(k;\tau,t)|\left|\frac{d}{dk}\psi(k;x)\right|dk.$$ By Lemma \ref{Kenlem}, we have $|\Phi(k;\tau,t)|\le \frac{c}{\sqrt{|t|}}$ for all $k\in [0,b]$, where $c$ only depends on $\varphi$ and is independent of free parameters and $b$. Moreover, we observe that $|\psi(b;x)|\le 1$ uniformly in $x$ and $b$, and $$\int_{0}^b\left|\frac{d}{dk}\psi(k;x)\right|dk={x}\int_{0}^be^{-{k}x}dk=(1-e^{-{b}x})\le 1$$ uniformly in $x$ and $b$.
	Hence, the result follows.
\end{proof}
It is immediate from the definition of $u_1$ and the above lemma that
\begin{equation}\label{Linftyestu1}
	|u_1(\cdot,t)|_{L_x^\infty(\mathbb{R}_+)}\lesssim \frac{1}{\sqrt{t}}|H_1|_{L_t^1(\mathbb{R})}.
\end{equation}
The following estimate is due to the fact that the Laplace transform is a bounded operator from $L_k^2(0,\infty)$ into $L_k^2(0,\infty)$:
\begin{equation}\label{L2estu1}
	|u_1(\cdot,t)|_{L_x^2(\mathbb{R}_+)}\lesssim |H_1|_{L_t^2(\mathbb{R})}.
\end{equation}
Interpolating between \eqref{Linftyestu1} and \eqref{L2estu1} via the Riesz-Thorin Interpolation Theorem, we get
\begin{equation}\label{Lrestu1}
	|u_1(\cdot,t)|_{L_x^r(\mathbb{R}_+)}\lesssim t^{-(\frac{1}{2}-\frac{1}{r})}|H_1|_{L_t^{r'}(\mathbb{R})},  \quad 2\le r\le \infty.
\end{equation}

The above estimate plays the key role in establishing the Strichartz estimates. This is rather standard. Indeed, \eqref{Lrestu1}, the admissibility condition \eqref{adm} and the Riesz potential inequalities imply
\begin{equation}\label{est1}
	\left|\int_0^{T'}\mathcal{T}_{B,1}(t-s)\theta(s)ds\right|_{L^\lambda(0,T';L^r(\mathbb{R}_+))}\lesssim |\theta|_{L^{\lambda'}(0,T';L^{r'}(\mathbb{R}_+))}
\end{equation} for any $\theta\in L_t^{\lambda'}(0,T';L_x^{r'}(\mathbb{R}_+))$. Now, let $\psi\in C_c([0,T');\mathcal{D}(\mathbb{R}_+))$. Then,
\begin{equation}\label{iden}
	\begin{split}
		&  \left|\int_{-\infty}^\infty(\mathcal{T}_{B,1}(t)h,\psi(t))_{L^2(\mathbb{R}_+)}dt\right| \\
		=  &\left|\int_{-\infty}^\infty\int_{0}^\infty\int_{-\infty}^{\infty} e^{-kx+ik^2t}\hat{H}_1(k)dk\overline{\psi}(x,t)dxdt\right|\\
		=  &   \lim_{b\rightarrow \infty} \left|\int_{0}^b\hat{H}_1(k)\overline{\int_{0}^{T'}\int_{0}^{\infty} e^{-kx-ik^2t}\psi(x,t)dxdt}dk\right|\\
		= & \lim_{b\rightarrow \infty} \left|\int_{-\infty}^\infty {H}_1(\tau)\overline{\int_{0}^{T'}\int_{0}^{\infty}  \overline{\ell}(\tau;x,t,b)\psi(x,t)dxdt} d\tau\right|.
	\end{split}
\end{equation}

We prove an auxiliary result now.
\begin{lem}\label{auxlem} If $\psi\in C_c([0,T');\mathcal{D}(\mathbb{R}_+))$ and $I(\tau;b)=\int_{0}^{T'}\int_{0}^{\infty}  \overline{\ell}(\tau;x,t,b)\psi(x,t)dxdt$, then
	\begin{equation}\label{aux1}
		\left|I(\tau;b)\right|_{L_\tau^2(\mathbb{R})} \le c|\psi|_{L^{\lambda'}(0,T';L^{r'}(\mathbb{R}_+))}.
	\end{equation}
\end{lem}
\begin{proof} We can rewrite the LHS of \eqref{aux1} as
	\begin{equation}
		\begin{split}
			\left|I(\tau;b)\right|_{L_\tau^2(\mathbb{R})} = \int_0^{T'}\int_0^\infty\psi(x,t)\int_0^{T'}\int_0^\infty\overline{\psi}(y,s)L(x,y,t,s;b)dydsdxdt,
		\end{split}
	\end{equation} where $L(x,y,t,s;b)=\int_{-\infty}^\infty\overline{\ell}(\tau;x,t,b){\ell}(\tau;y,s,b)d\tau.$  By using the finite line Fourier transform, we have
	\begin{equation}
		\begin{split}
			L(x,y,t,s;b)  & = \int_{0}^b\int_{-\infty}^\infty\int_{0}^b  e^{-kx-ik^2t+ik\tau}e^{-\tilde{k}y+i\tilde{k}^2s-i\tilde{k}\tau}dkd\tilde{k}d\tau\\
			& = 2\pi\int_{0}^b e^{-k(x+y)-ik^2(t-s)}dk.
		\end{split}
	\end{equation}
	By using the same arguments in Lemma \ref{vander}, we deduce that $$|L(x,y,t,s;b)|\le \frac{c}{\sqrt{|t-s|}},\quad t\neq s, \text{ uniformly in }x,y,b\in\mathbb{R}_+.$$ The above estimate and arguments used in the proof of \eqref{est1} give the result.
\end{proof}
Using Lemma \ref{auxlem}, we find that the RHS of \eqref{iden} is bounded by
\begin{equation}\label{RHSbnd}
	c|H_1|_{L^2(\mathbb{R})}|\psi|_{L_t^{\lambda'}(0,T';L_x^{r'}(\mathbb{R}_+))}.
\end{equation}
Hence, by duality we establish the case $s=0$:
\begin{equation}\label{szero}|\mathcal{T}_{B,1}(t)h|_{L_t^\lambda(0,T';L_x^r(\mathbb{R}_+))}\le c|H_1|_{L^2(\mathbb{R})}.\end{equation}
\begin{rem}\label{difrem}
	The integral representation formula on $\partial D_+$ obtained via the Fokas method has the remarkable property that one can easily differentiate with respect to $x$ and $t$.
\end{rem}
In view of Remark \ref{difrem}, differentiating $u_1$ in $x$ merely brings a factor of a scalar multiple of $k$ into the integrand. Therefore, the above arguments can be repeated for $\partial_x[\mathcal{T}_{B,1}(t)h]$, and one obtains
\begin{equation}\label{sone}|\partial_x\mathcal{T}_{B,1}(t)h|_{L_t^\lambda(0,T';L_x^r(\mathbb{R}_+))}\le c|\tilde{H}_1|_{L^2(\mathbb{R})},\end{equation} where $\tilde{H}_1(k)\equiv -k\hat{H}_1(k)$ for $k\in \mathbb{R}$. From \eqref{szero} and \eqref{sone}, we establish the $s=1$ case:
\begin{equation}\label{h1est}
	|\mathcal{T}_{B,1}(t)h|_{L_t^\lambda(0,T';H_x^{1,r}(\mathbb{R}_+))}\le c|H_1|_{H^1(\mathbb{R})}
\end{equation}

Now, we interpolate between \eqref{szero} and \eqref{h1est} and obtain
\begin{equation}\label{hsest}
	|\mathcal{T}_{B,1}(t)h|_{L_t^\lambda(0,T';H_x^{s,r}(\mathbb{R}_+))}\le c|H_1|_{H^s(\mathbb{R})}
\end{equation} for $0\le s\le 1$. Finally, we can iterate and derive the same estimate for all $s\ge 0$.

Returning to the original boundary input is easy. Namely, $H_1$ and $h$ are related via the estimate
\begin{equation}\label{H1hrel}
	\begin{split}
		&  |H_1|_{H^s(\mathbb{R})}^2 = \int_{-\infty}^\infty (1+k^2)^s|\hat{H}_1(k)|^2dk\\
		&=\frac{1}{\pi^2}\int_{0}^\infty (1+k^2)^sk^2|\hat{h}(k^2)|^2dk=\frac{1}{2\pi^2}\int_{0}^\infty (1+\tau)^s\sqrt{\tau}|\hat{h}(\tau)|^2d\tau\\
		& \le c\int_{-\infty}^\infty (1+\tau^2)^{\frac{2s+1}{4}}|\hat{h}(\tau)|^2d\tau=c|h|_{H_t^{\frac{2s+1}{4}}(\mathbb{R})}^2.
	\end{split}
\end{equation}
It follows from \eqref{hsest} and \eqref{H1hrel} that
$|u_1|_{L_t^\lambda(0,T';H_x^{s,r}(\mathbb{R}_+))}\lesssim |h|_{H_t^{\frac{2s+1}{4}}(\mathbb{R})}.$
\subsubsection{Analysis on the real axis: ibvp to Cauchy switch} Here, we set a function $H_2$ as the inverse Fourier transform of
\begin{equation}\label{H2hat}
	\hat{H}_2(k)\equiv \left\{
	\begin{array}{ll}
		2k\hat{h}(-k^2), & \hbox{if $k\ge 0$;} \\
		0, & \hbox{otherwise}.
	\end{array}
	\right.
\end{equation} Then, $u_2$ is rewritten as
\begin{equation}\label{u2rep}
	u_2(x,t)=\frac{1}{2\pi}\int_{-\infty}^{\infty}e^{ikx-ik^2t}\hat{H}_2(k)dk,
\end{equation} where $x\in \mathbb{R}_+.$
Observe that the above formula makes sense even for negative $x$. Therefore, extending \eqref{u2rep} to all $x\in\mathbb{R}$ and comparing it with \eqref{vform}, we see that $u_2$, denoted same, becomes the solution of a Cauchy problem which reads as
\begin{align}
	&\partial_t u_2+Pu_2=0, \quad (x,t)\in\mathbb{R}\times (0,T'),\label{maineqlinHC2}\\
	&u_2(x,0)=H_2(x).\label{initlinHC2}
\end{align}

Although, $H_2$ originally depends on the time variable $t$, the above trick allows us to write a Cauchy problem in which the dummy variable of $H_2$ is switched with the spatial variable $x$. In some sense, the ibvp with time dependent boundary input is translated into an initial value problem (ivp) on the whole line.  The advantage is that Strichartz estimates for the ivp are well-known. Indeed, by the Cauchy theory (see Theorem \ref{Cauchythm}), we have
\begin{equation}\label{switch1}
	|u_2|_{L_t^\lambda(0,T; H_x^{s,r}(\mathbb{R}))}\lesssim |H_2|_{H^{s}(\mathbb{R})}.
\end{equation} Note also that $H_2$ is controlled by $h$ via
\begin{equation}\label{H2hrel}
	\begin{split}
		&  |H_2|_{H^s(\mathbb{R})}^2 = \int_{-\infty}^\infty (1+k^2)^s|\hat{H}_2(k)|^2dk\\
		&=\int_{0}^\infty (1+k^2)^s4k^2|\hat{h}(-k^2)|^2dk=\frac{1}{2}\int_{0}^\infty (1+\tau)^s\sqrt{\tau}|\hat{h}(-\tau)|^2d\tau\\
		& \le c\int_{-\infty}^\infty (1+\tau^2)^{\frac{2s+1}{4}}|\hat{h}(\tau)|^2d\tau=c|h|_{H_t^{\frac{2s+1}{4}}(\mathbb{R})}^2.
	\end{split}
\end{equation}
It follows from \eqref{switch1} and \eqref{H2hrel} and restricting $u_2$ back to the half line that
$$|u_2|_{L_t^\lambda(0,T; H_x^{s,r}(\mathbb{R}_+))}\lesssim |h|_{H_t^{\frac{2s+1}{4}}(\mathbb{R})}.$$
\subsection{Proof of Theorem \ref{mainthm2}}\label{pf2}
\subsubsection{Homogeneous estimate}
The solution splits as
\begin{equation}\label{decompu20}
	\begin{split}
		u(x,t) & = -\frac{i}{\pi}\int_{\partial D^+}e^{ikx-ik^2t}\tilde{h}(k^2,T')dk\\
		&= \frac{1}{\pi}\int_\infty^0e^{-kx+ik^2t}\hat{h}(k^2)dk-\frac{i}{\pi}\int_{0}^\infty e^{ikx-ik^2t}\hat{h}(-k^2)dk\\
		&\equiv  u_1(x,t)+u_2(x,t).
	\end{split}
\end{equation}
We set \begin{equation}\label{H1hat20}
	\hat{H}_1(k)\equiv \left\{
	\begin{array}{ll}
		-\frac{1}{\pi}\hat{h}(k^2), & \hbox{if $k\ge 0$;} \\
		0, & \hbox{otherwise}
	\end{array}
	\right.
\end{equation} and
\begin{equation}\label{H1hat30}
	\hat{H}_2(k)\equiv \left\{
	\begin{array}{ll}
		-2i\hat{h}(-k^2), & \hbox{if $k\ge 0$;} \\
		0, & \hbox{otherwise}.
	\end{array}
	\right.
\end{equation}
From the arguments in the proof of Theorem \ref{mainthm} with $\hat{H}_i$, $i=1,2$ as above, we have for $s\in \mathbb{N}_0$ that
\begin{equation}\label{switch120}
	|u_i|_{L_t^\lambda(0,T; \dot{H}_x^{s,r}(\mathbb{R}_+))}\lesssim |H_i|_{\dot{H}^{s}(\mathbb{R})},\quad i=1,2.
\end{equation}
Observe that
\begin{equation}\label{H2hrelN0}
	\begin{split}
		&  |H_1|_{\dot{H}^s(\mathbb{R})}^2 = \int_{-\infty}^\infty k^{2s}|\hat{H}_1(k)|^2dk\lesssim \int_{0}^\infty k^{2s}|\hat{h}(k^2)|^2dk\\&\lesssim\int_{0}^\infty \frac{\tau^s}{\sqrt{\tau}}|\hat{h}(\tau)|^2d\tau
		\lesssim\int_{-\infty}^\infty |\tau|^{s-\frac{1}{2}}|\hat{h}(\tau)|^2d\tau=|h|_{\dot{H}_t^{\frac{2s-1}{4}}(\mathbb{R})}^2.
	\end{split}
\end{equation}
A similar estimate also holds for $H_2$.
\subsubsection{Inhomogeneous estimate} We first consider the case $s>1/2$. We define a new function $h_e$ of mean zero with the properties $h_e|_{[0,T')}=h|_{[0,T')}$, $\supp h_e\subset [0,2T'+1)$, $|h_e|_{H_t^{\frac{2s-1}{4}}(\mathbb{R})}\lesssim |h|_{H_t^{\frac{2s-1}{4}}(\mathbb{R})}$ so that $H(t)\equiv \int_{-\infty}^{t}h_e(s)ds$ has also support in $[0,2T'+1)$ with $|H|_{H_t^{\frac{2s+3}{4}}(\mathbb{R})}\lesssim (1+T')|h|_{H_t^{\frac{2s-1}{4}}(\mathbb{R})},$ see for instance \cite[Lemma 2.1]{Batal16} for such construction.  Then, we solve the Neumann problem with $h_e$ over the interval $[0,2T'+1)$.  The corresponding solution is given by
\begin{equation}\label{decompu2}
	\begin{split}
		u(x,t) & = -\frac{i}{\pi}\int_{\partial D^+}e^{ikx-ik^2t}\tilde{h}_e(k^2,2T'+1)dk\\
		&= \frac{1}{\pi}\int_\infty^0e^{-kx+ik^2t}\hat{h}_e(k^2)dk-\frac{i}{\pi}\int_{0}^\infty e^{ikx-ik^2t}\hat{h}_e(-k^2)dk\\
		&\equiv  u_1(x,t)+u_2(x,t).
	\end{split}
\end{equation}
We set \begin{equation}\label{H1hat2}
	\hat{H}_1(k)\equiv \left\{
	\begin{array}{ll}
		-\frac{1}{\pi}\hat{h}_e(k^2), & \hbox{if $k\ge 0$;} \\
		0, & \hbox{otherwise}
	\end{array}
	\right.
\end{equation} and
\begin{equation}\label{H1hat3}
	\hat{H}_2(k)\equiv \left\{
	\begin{array}{ll}
		-2i\hat{h}_e(-k^2), & \hbox{if $k\ge 0$;} \\
		0, & \hbox{otherwise}.
	\end{array}
	\right.
\end{equation}
Now, we repeat the arguments in the proof of Theorem \ref{mainthm} by taking $\hat{H}_i$, $i=1,2$ as in \eqref{H1hat2} and \eqref{H1hat3}, respectively.
Namely, we have
\begin{equation}\label{switch12}
	|u_i|_{L_t^\lambda(0,T; H_x^{s,r}(\mathbb{R}_+))}\lesssim |H_i|_{H^{s}(\mathbb{R})},\quad i=1,2.
\end{equation}
Observe that
\begin{equation}\label{H2hrelN}
	\begin{split}
		&  |H_1|_{H^s(\mathbb{R})}^2 = \int_{-\infty}^\infty (1+k^2)^s|\hat{H}_1(k)|^2dk\lesssim \int_{0}^\infty (1+k^2)^s|\hat{h}_e(-k^2)|^2dk\\&\lesssim\int_{0}^\infty \frac{(1+\tau)^s}{\sqrt{\tau}}|\hat{h}_e(-\tau)|^2d\tau
		\lesssim\int_{-\infty}^\infty (1+\tau^2)^{\frac{s}{2}}|\tau^2|^{-\frac{1}{4}}|\hat{h}_e(\tau)|^2d\tau.
	\end{split}
\end{equation}
Using $\frac{d}{dt}H=h_e \Rightarrow |\tau||\hat{H}(\tau)|=|\hat{h}_e(\tau)|$, the term at the right hand side of \eqref{H2hrelN} is estimated as
\begin{equation}\label{rhsest}
	\begin{split}\int_{-\infty}^\infty (1+\tau^2)^{\frac{s}{2}}|\tau|^{-\frac{1}{2}}|\hat{h}_e(\tau)|^2d\tau\le \int_{-\infty}^\infty (1+\tau^2)^{\frac{s}{2}}|\tau^2|^{\frac{3}{4}}|\hat{H}(\tau)|^2d\tau\\
		=|H|_{H_t^{\frac{2s+3}{4}}(\mathbb{R})}^2\lesssim (1+T')^2|h|_{H_t^{\frac{2s-1}{4}}(\mathbb{R})}^2.
	\end{split}
\end{equation} Hence, $|H_1|_{H^s(\mathbb{R})}\lesssim (1+T')|h|_{H_t^{\frac{2s-1}{4}}(\mathbb{R})}.$ A similar estimate also holds for $H_2$.\\
Next, we consider the case $s<1/2$.
For $s=0$, $\frac{2s-1}{4}=-\frac{1}{4}$. We recall the following lemma.
\begin{lem}[\cite{Holmer}]\label{hlm_lem}
	Let $\theta \in C_0^{\infty}(\Real)$ and $0\leq \alpha < 1/2$. If $h\in H^{-\alpha}$ then
	\begin{equation*}
		|\theta h|_{\dot{H}^{-\alpha}}\leq c(\theta,\alpha) | h|_{\dot{H}^{-\alpha}},
	\end{equation*}
	where the constant of inequality depends on $\alpha$ and the size of the support of $\theta$.
\end{lem}
We first prove \eqref{StrEst03} for $s=0$ in which case, $h\in H_t^{-\frac{1}{4}}$ with $\supp h\subset [0,T')$. Then by the homogeneous estimate \eqref{StrEst02} and Lemma \ref{hlm_lem} it follows that for a cutoff function $\theta$ which satisfies $\theta\equiv 1$ on $[0,T']$, we have
\begin{equation}
	|u|_{L^{\lambda}_t(0,T';L_x^r(\Real_+))}\lesssim |h|_{\dot{H_t}^{-\frac{1}{4}}(\Real)}=|\theta h|_{\dot{H_t}^{-\frac{1}{4}}(\Real)}
	\leq c_{T'} |h|_{H_t^{-\frac{1}{4}}(\Real)}.
\end{equation}
For $s=1$ we already established  inhomogeneous estimate \eqref{StrEst03} so we can interpolate to obtain the desired result for $s\in [0,1]$.

\subsection{Temporal regularity of spatial traces for Cauchy problems}
In this section, we overview two theorems regarding the time trace estimates of solutions of the homogeneous and nonhomogeneous Cauchy problems, respectively.
\begin{thm}\label{neumann_trace_homo}
	Let $s\geq 0$ and $y_0^*\in H^s_x(\Real)$. Then $v(t)=e^{-tP}y_0^*$ defines a solution to \eqref{vform} such that $\partial_xv\in C(\Real_x,H_t^{\frac{2s-1}{4}}(0,T))$ for $T\in (0,\infty]$ and
	\begin{equation}\label{hcp_test}
		\sup_{x\in \Real}|\partial_xv(x,\cdot)|_{H_t^{\frac{2s-1}{4}}(0,T)}\leq \sqrt{2}|y_0^*|_{H_x^s}.
	\end{equation}
\end{thm}
\begin{proof}
	By the solution representation of the homogeneous Cauchy problem we have
	\begin{align*}
		\partial_xv(x,t) & = \frac{i}{2\pi}\int_{\Real}e^{i\xi x-i\xi^2t}\xi \widehat{y_0^*}(\xi) d\xi \\
		& = \frac{i}{2\pi}\int_0^{\infty}e^{i\xi x-i\xi^2t}\xi \widehat{y_0^*}(\xi) d\xi +\frac{i}{2\pi}\int^0_{-\infty}e^{i\xi x-i\xi^2t}\xi \widehat{y_0^*}(\xi) d\xi \\
		& \eqqcolon I(x,t)+II(x,t).
	\end{align*}
	First we estimate $II$. By the change of variable $\xi=-\sqrt{\tau}$ we rewrite $II$ as
	\begin{equation}
		II(x,t)= \frac{i}{4\pi}\int_0^{\infty}e^{-i\sqrt{\tau}x-i\tau t}\widehat{y_0^*}(-\sqrt{\tau})d\tau.
	\end{equation}
	Then $II$ is the inverse (in time) Fourier transform of the function
	\begin{equation}
		\widehat{II}^{(t)}(x,\tau) \coloneqq \frac{i}{2}\chi_{[0,\infty)}(\tau)e^{-i\sqrt{\tau}x}\widehat{y_0^*}(-\sqrt{\tau}).
	\end{equation}
	So by the definition of the Sobolev norm,
	\begin{align}\label{b2}
		\nonumber |II(x,\cdot)|^2_{H_t^{\frac{2s-1}{4}}} &\leq \int_{\Real} (1+\tau^2)^{\frac{2s-1}{4}}|\widehat{II_1}^{(t)}(x,\tau)|^2d\tau\\ \nonumber
		& = \frac{1}{4} \int_0^{\infty} (1+\tau^2)^{\frac{2s-1}{4}}|\widehat{y_0^*}(-\sqrt{\tau})|^2d\tau\\
		& = \frac{1}{2}\int_{-\infty}^{0}(1+\xi^4)^\frac{2s-1}{4}|\widehat{y_0^*}(\xi)|^2\xi d\xi \leq \frac{1}{2}|y_0^*|^2_{H^s}.
	\end{align}
	By a similar argument, one can show that
	\begin{equation}\label{b3}
		|I(x,\cdot)|_{H_t^{\frac{2s-1}{4}}}\leq \frac{1}{\sqrt{2}}|y_0^*|_{H_x^s}.
	\end{equation}
	Then \eqref{b2} and \eqref{b3} together imply the desired result.
\end{proof}
Next, we estimate time traces of the solution of the nonhomogeneous Cauchy problem.
\begin{thm}\label{neumann_trace_inhomo}
Let $(\lambda,r)$ be Schrödinger admissible, $m\in \{0,1\}$ and $z$ be defined by \eqref{zform}. Then, the following properties hold:
\begin{enumerate}
    \item Let $f^*\in L_t^{\lambda'}\dot{H}_x^{s,r'}$, $s\in \mathbb{Z}$, then
    $\partial_x^mz\in C(\Real;\dot{H}_t^{\frac{2s+1-2m}{4}}(0,T))$
    for $T\leq \infty$ and
    \begin{equation}\label{homo_est}
        \sup_x|\partial_x^mz(x,\cdot)|_{\dot{H}_t^{\frac{2s+1-2m}{4}}(0,T)}\lesssim |f^*|_{L_t^{\lambda'}\dot{H}_x^{s,r'}}.
    \end{equation}
    \item Let $f^*\in L_t^{\lambda'}H_x^{s,r'}$, $s\in [0,\infty)$, then $\partial_x^mz\in C(\Real;H_t^{\frac{2s+1-2m}{4}}(0,T))$ for $T<\infty$ and
    \begin{equation}\label{inhomo_est}
        \sup_x|\partial_x^mz(x,\cdot)|_{H_t^{\frac{2s+1-2m}{4}}(0,T)}\lesssim (1+c_T)|f^*|_{L_t^{\lambda'}H_x^{s,r'}}
    \end{equation}
\end{enumerate}

\end{thm}
\begin{rem}
  The homogeneous and inhomogeneous time trace estimates in the above theorem can be considered as extensions of estimates in \cite{Holmer} to a larger range of $s$. Here, we obtain the estimates for a large range of $s$ by making use of the relation between the spatial and temporal fractional derivatives of the solution.  Moreover, we refrain from differentiating the solution formula in time to obtain higher order time trace estimates. This generally cause additional time trace terms appearing at the right hand side of \eqref{inhomo_est}, see for instance \cite{Bona18}. Instead, we use an iterative argument whose each step relies on boundedness of $|\partial_x^mz(x,\cdot)|_{H_t^{\frac{1}{4}}(0,T)}$.
\end{rem}

\begin{rem}
	${L_t^{\lambda'}H_x^{s,r'}}$ norms of $f^*$ at the right hand side of the estimates in Theorem \ref{neumann_trace_inhomo} can be replaced by ${L_t^{\lambda'}(0,T;H_x^{s,r'}})$ norms. The latter are more convenient for nonlinear applications.
\end{rem}

\begin{lem}[\cite{HolmerThesis, Ken93}]\label{1}
    For $f\in \mathcal{D}_{\otimes} \equiv  \big\{h\,|\,h(x,t)=\sum_{i=1}^Nh_i^1(x)h_i^2(t), h_i^1,h_i^2\in C_0^{\infty}(\Real)\big\}$,
\begin{equation}\label{sol}
    \begin{split}
        2z(x,t) &= \int_{\Real} e^{-(t-t')P}f(x,t')dt'- 2\int_{-\infty}^0 e^{-(t-t')P}f(x,t')dt'\\
        & -\frac{i}{\pi}\int_{\Real} e^{it\tau} \Big[ \lim_{\epsilon \rightarrow 0^+} \int_{\epsilon<|\xi^2+\tau|<\frac{1}{\epsilon}} e^{i x\xi}\frac{\widehat{f}(\xi,\tau)}{\xi^2+\tau} d \xi \Big] d \tau.
    \end{split}
\end{equation}
\end{lem}
\begin{proof} This lemma is an analogue of a proposition that was given by Kenig, et al. in \cite{Ken93} for the Korteweg-de Vries equation.  Its adaptation to the Schrödinger equation was given in Holmer \cite{HolmerThesis} without proof. We give a proof for completeness.  We write the last term in \eqref{sol}  as
\begin{equation}\label{last_term}
    -\frac{i}{\pi} \int_{\Real} e^{ix\xi} \lim_{\epsilon \downarrow 0} \int_{\epsilon<|\xi^2+\tau|<\frac{1}{\epsilon}} e^{it\tau}\frac{\widehat{f}(\xi,\tau) }{\xi^2+\tau} d\tau d\xi.
\end{equation}
Now by the Fourier transform of the signum function and a change of variable we have
\begin{align*}
    \int_{\epsilon<|\xi^2+\tau|<\frac{1}{\epsilon}} e^{it\tau}\frac{\widehat{f}(\xi,\tau) }{\xi^2+\tau} d\tau &= \frac{i}{2}\int_{\epsilon<|\xi^2+\tau|<\frac{1}{\epsilon}} e^{it\tau} \widehat{f}(\xi,\tau)\Big(\int_{\Real} e^{-i(\xi^2+\tau)t'} sgn(t') dt' \Big) d\tau\\
    &=  \frac{i}{2}\int_{\epsilon<|\xi^2+\tau|<\frac{1}{\epsilon}} e^{it\tau} \widehat{f}(\xi,\tau)\Big(\int_{\Real} e^{-i(\xi^2+\tau)(t-t')} sgn(t-t') dt' \Big) d\tau\\
    &=  \frac{i}{2}\int_{\epsilon<|\xi^2+\tau|<\frac{1}{\epsilon}} e^{-i\xi^2t} \widehat{f}(\xi,\tau)\Big(\int_{\Real} e^{i(\xi^2+\tau)t'} sgn(t-t') dt' \Big) d\tau\\
    &=\frac{i}{2} \int_{\Real}e^{-i\xi^2(t-t')}sgn(t-t')\int_{\epsilon<|\xi^2+\tau|<\frac{1}{\epsilon}} e^{it'\tau} \widehat{f}(\xi,\tau)d\tau dt'
\end{align*}

\begin{equation*}
    \Rightarrow \lim_{\epsilon \downarrow 0 }\int_{\epsilon<|\xi^2+\tau|<\frac{1}{\epsilon}} e^{it\tau}\frac{\widehat{f}(\xi,\tau) }{\xi^2+\tau} d\tau = \pi i \int_{\Real}e^{-i\xi^2(t-t')} sgn(t-t')\widehat{f}^{(x)}(\xi,t')dt'
\end{equation*}
and hence \eqref{last_term} is equal to
\begin{align*}
    & \int_{\Real} e^{ix\xi} \int_{\Real} e^{-i\xi^2(t-t')}sgn(t-t')\widehat{f}^{(x)}(\xi,t')d\xi dt'\\
    &= \int_{\Real} sgn(t-t')\int_{\Real} e^{ix\xi-i\xi^2(t-t')}\widehat{f}^{(x)}(\xi,t')d\xi dt'\\
    &= \int_{\Real}sgn(t-t') e^{-(t-t')P}f(\cdot,t')dt'\\
    &= 2\int_{-\infty}^t e^{-(t-t')P}f(\cdot, t')dt'-\int_{\Real}e^{-(t-t')P}f(\cdot, t')dt'\\
    &= 2\int_{-\infty}^0e^{-(t-t')P}f(\cdot,t')dt'-\int_{\Real}e^{-(t-t')P}f(\cdot,t')dt'+2z(x,t).
\end{align*}
\end{proof}

\begin{rem} $\mathcal{D}_{\otimes}$ is dense in $L_t^pL_x^p$ for any $p,q\ge 1$ \cite{Ken93}.
\end{rem}

\begin{lem}[\cite{HolmerThesis}]\label{2}
Let $x\in \Real$. Then,
\begin{equation}\label{tau_neg}
    \lim_{\epsilon \downarrow 0 }\int_{\epsilon<|\xi^2+\tau|<\frac{1}{\epsilon}}e^{ix\xi}\frac{1}{\xi^2+\tau}d\xi = -\frac{\pi \sin(|x|\sqrt{-\tau})}{\sqrt{-\tau}}, \tau <0
\end{equation} and
\begin{equation}\label{tau_pos}
    \int_{\Real}e^{ix\xi} \frac{1}{\xi^2+\tau} d\xi =\pi \frac{e^{-|x|\sqrt{\tau}}}{\sqrt{\tau}}, \tau>0.
\end{equation}
\end{lem}
\begin{proof} It is stated in \cite{HolmerThesis} (for $\tau=1$) that two identities given above  can be proven  by making use of partial fraction decomposition and delta function, respectively. We give an alternate and complex analytic proof here. Let us first consider the case $x>0$ and $\tau<0$. Set the complex valued function $f(z)=e^{ixz}\frac{1}{z^2+\tau}$
over the contour $\Gamma_{\epsilon}$ shown in Figure \ref{gamma}.  By the residue theorem $\int_{\Gamma_{\epsilon}}f=0$ and since $x>0$, Jordan's Lemma implies $\lim_{\epsilon \downarrow 0}\int_{\gamma_{\epsilon}}f=0$.
We also have
\begin{equation}
    \int_{D_{\epsilon}}f= -i\pi Res(f,\sqrt{-\tau})=-\frac{i\pi}{2}\frac{e^{ix\sqrt{-\tau}}}{\sqrt{-\tau}}\text{ and }
    \int_{C_{\epsilon}}f= -i\pi Res(f,-\sqrt{-\tau})=\frac{i\pi}{2}\frac{e^{-ix\sqrt{-\tau}}}{\sqrt{-\tau}}.
\end{equation}
Thus for $x>0$ we obtain $$\lim_{\epsilon \downarrow 0 }\int_{\epsilon<|\xi^2+\tau|<\frac{1}{\epsilon}}e^{ix\xi}\frac{1}{\xi^2+\tau}d\xi = -\frac{\pi \sin(x\sqrt{-\tau})}{\sqrt{-\tau}}.$$
For $x<0$, we can repeat the same argument for the contour $\Gamma'_{\epsilon}$ in Figure \ref{0gammaprime} to obtain the same expression, proving \eqref{tau_neg}. \eqref{tau_pos} can be shown by using similar complex analytic arguments.

\begin{figure}[H]
	\centering
	\begin{minipage}{.5\textwidth}
		\centering
		\includegraphics[width=\linewidth]{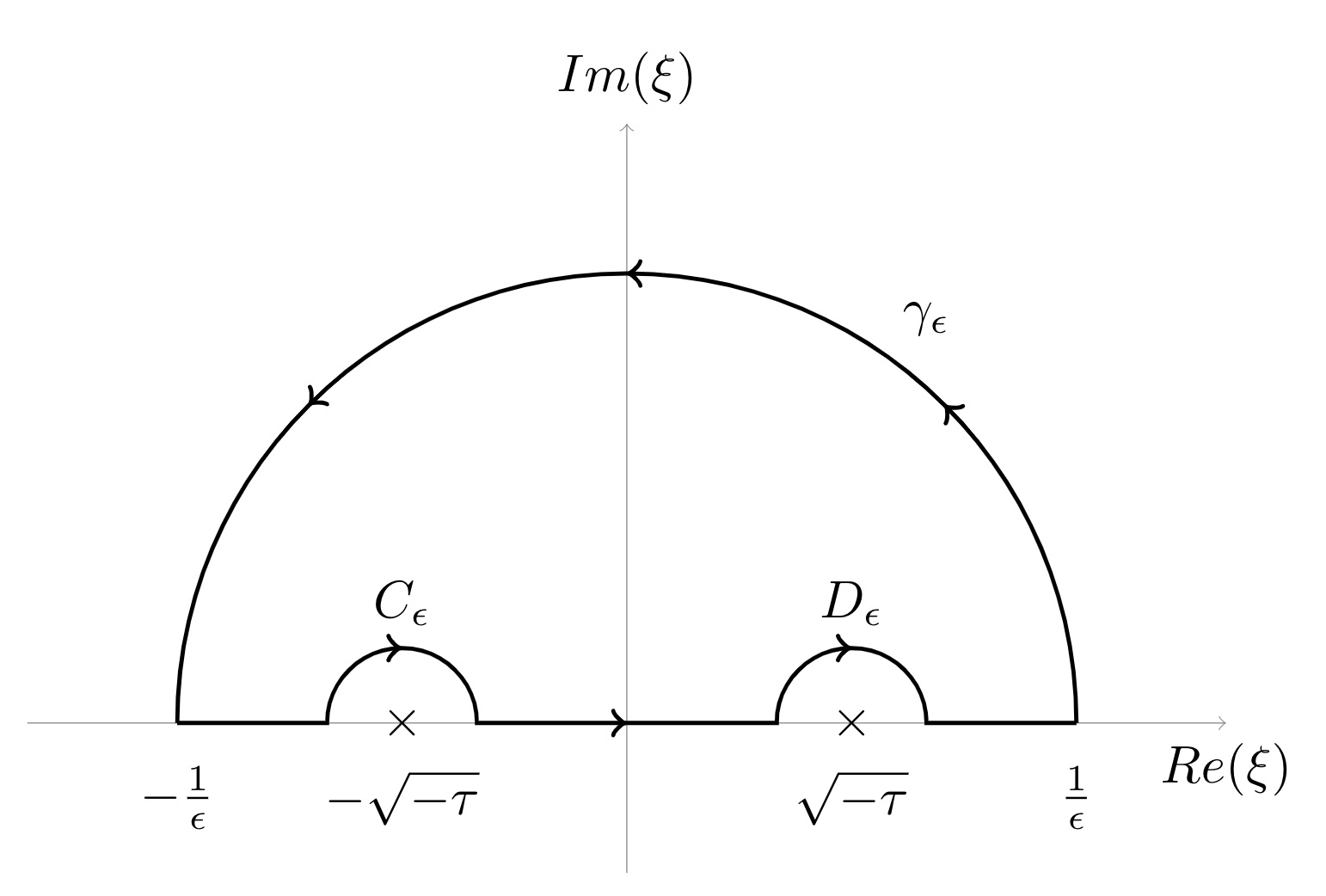}
\captionof{figure}{The contour $\Gamma_{\epsilon}$}
		\label{gamma}
	\end{minipage}%
	\begin{minipage}{.5\textwidth}
		\centering
		\includegraphics[width=\linewidth]{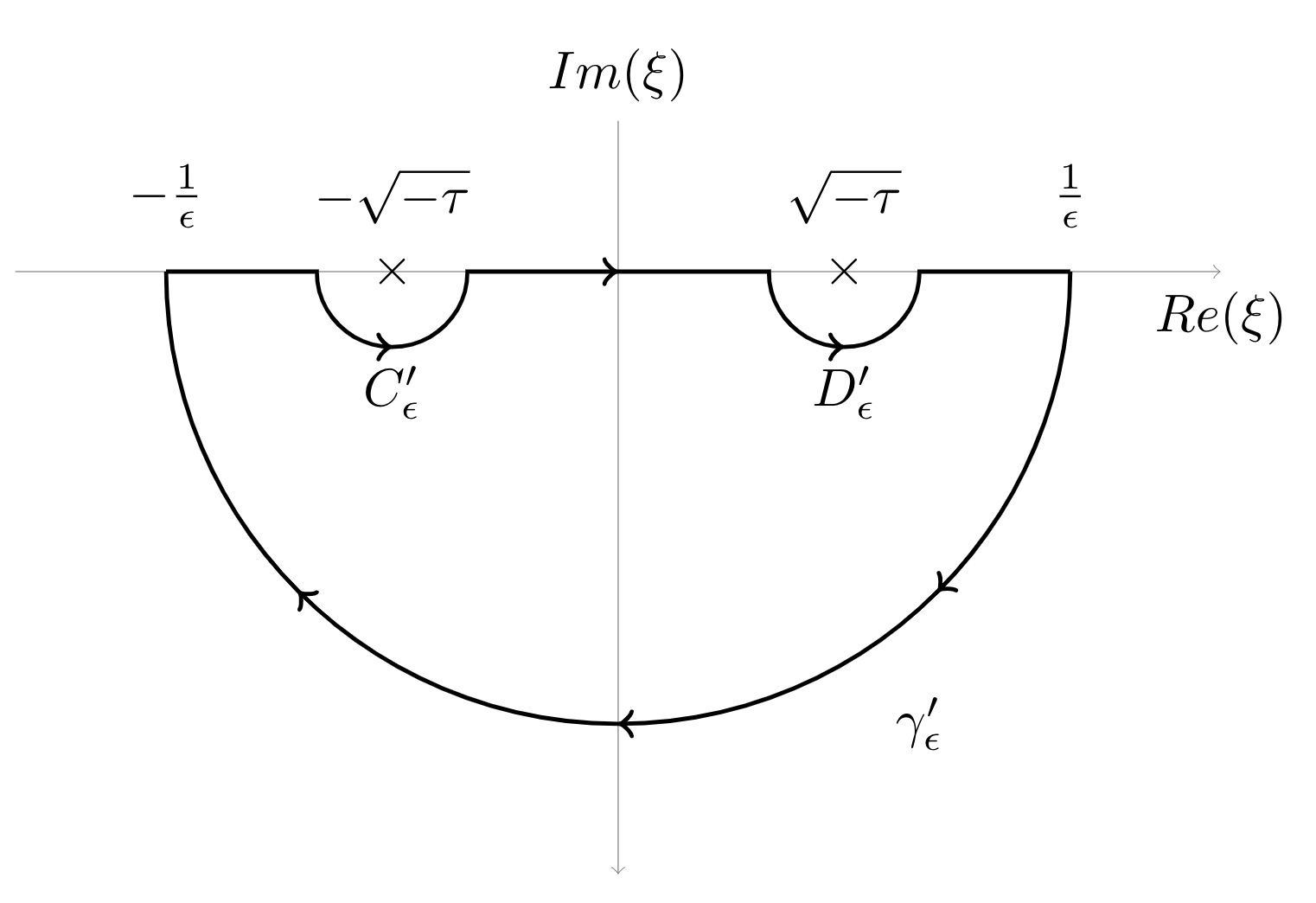}
		\captionof{figure}{The contour $\Gamma'_{\epsilon}$}
		\label{0gammaprime}
	\end{minipage}
\end{figure}
\end{proof}
The following lemma will be useful.
\begin{lem}[\cite{Ben-Artzi}]\label{3}
Suppose $\alpha \in \Real$. Then for $t\neq 0$
\begin{equation}
    \Big| \int_0^{\infty} e^{i\alpha \xi \pm i\xi^2t}d\xi \Big|\lesssim \sqrt{|t|},
\end{equation} where the constant of the inequality is independent of $\alpha$ and $t$.
\end{lem}
Now, we are ready to prove Theorem \ref{neumann_trace_inhomo}.
\begin{proof}[Proof of Theorem \ref{neumann_trace_inhomo}]
The proof of this theorem is largely due to Holmer \cite{HolmerThesis}, here we only mention  necessary modifications to allow a larger range of $s$. In particular, we make an explicit use of the relation between the spatial and temporal derivatives of the solution. Lemma \ref{1} gives
\begin{equation}
    \begin{split}
    D_t^{\frac{2s+1}{4}}z(x_0,t) &= \frac{1}{2} D_t^{\frac{2s+1}{4}}\int_{\Real} e^{-(t-t')P}f(x_0,t')dt'
    - D_t^{\frac{2s+1}{4}}\int_{\infty}^0e^{-(t-t')P}f(x_0,t')dt'\\
    &-\frac{i}{2\pi} \lim_{\epsilon \downarrow 0} \int \int_{\epsilon<|\xi^2+\tau|<\frac{1}{\epsilon}} e^{it\tau+i\xi x_0}|\tau|^{\frac{2s+1}{4}}\frac{\widehat{f}(\xi,\tau)}{\xi^2+\tau}d\xi d\tau\\
    & = I + II + III.
    \end{split}
\end{equation}
In order to estimate $I$, let $g\in L^2$ s.t. $|g|_{L^2}\leq 1$, then
\begin{align*}
|I^{x_0}|_{L_t^2} & \eqsim \sup_{|g|_{L^2}\leq 1} \Big| \int_{\Real} D_t^{\frac{2s+1}{4}}\Big( \int_{\Real}e^{-(t-t')P}f(x_0,t')dt' \Big) \overline{g(t)}dt \Big|\\
& = \sup_{|g|_{L^2}\leq 1} \Big| \int_{\Real} D_t^{\frac{s}{2}}\Big( \int_{\Real}e^{-(t-t')P}f(x_0,t')dt' \Big) \overline{D_t^{\frac{1}{4}}g(t)}dt \Big|.
\end{align*}
By definition of $e^{-(t-t')P}$ and Fourier transform characterization of fractional derivative, one has
    \begin{align*}
        & D_t^{\frac{s}{2}}\int_{t'}e^{-(t-t')P}f(x_0,t')dt'  \eqsim  \int_{t'} e^{-(t-t')P}D_x^sf(x,t') dt'.
    \end{align*}
Rewriting $I^{x_0}$ in view of the above identity and using arguments similar to those in  \cite{HolmerThesis}, one can prove that  $|I^{x_0}|_{L^2_t}\lesssim |H^{x_0}|_{L^2}$, where $H^{x_0}(x) =  \int_{\Real}e^{t'P}D_x^sf(x_0-x,t')dt'.$
Furthermore, $|H^{x_0}|_{L^2}\lesssim |f|_{L_t^{\lambda'}\dot{H}_x^{s,r'}}$.
By replacing the integrals where $t'\in \Real$ with integrals where $t'\in (0,\infty)$ in the analysis of $I$, it also follows that $|II^{x_0}|_{L^2_t}\lesssim |f|_{L_t^{\lambda'}\dot{H}_x^{s,r'}}.$
Next, we prove that $|III|_{L_t^2}\lesssim |f|_{L_t^{\lambda'}\dot{H}_x^{s,r'}}$. We start with the case $(\lambda,r)=(\infty,2)$. By Plancherel's theorem and Minkowski's inequality we have $$|III|_{L^2_t}\leq \int_t(F_1(t)+F_2(t))dt,$$ where
\begin{equation*}
    III_1(t)= \Bigg( \int_{\tau} |\tau|^{s+\frac{1}{2}} \Big| \int_{-\infty}^0  e^{ix\xi} \frac{\widehat{f}(\xi,t)}{\xi^2+\tau} d\xi \Big|^2 \Bigg)^{\frac{1}{2}}\text{ and }
    III_2(t) = \Bigg( \int_{\tau} |\tau|^{s+\frac{1}{2}} \Big| \int_0^{\infty}  e^{ix\xi} \frac{\widehat{f}(\xi,t)}{\xi^2+\tau} d\xi \Big|^2 \Bigg)^{\frac{1}{2}}.
\end{equation*}
Utilizing the fact that $|\tau|^{s+\frac{1}{2}}$ is an $A_2$ weight (see e.g., \cite[Chapter 5]{SteinHA}), a change of variable and Plancherel's theorem, one gets
\begin{equation*}
    III_i(t) \lesssim |f(\cdot,t)|_{\dot{H}_x^s}, i=1,2.
\end{equation*}
Now we prove the case $(\lambda,r)= (4,\infty) $ for $s\in \mathbb{Z}$. Let $g$ denote the function with the Fourier transform
$\widehat{g}(\xi,\tau)=\frac{1}{\xi^2+\tau}.$ We rewrite $III$ as
\begin{equation*}
    \begin{split}
        III & = \int_{\tau}e^{it\tau}|\tau|^{\frac{2s+1}{4}} \Bigg( \int_{\xi} e^{ix_0\xi} \mathcal{F}_x \Big[ \widehat{f}^{(t)}(x,\tau) *\widehat{g}^{(t)}(x,\tau) \Big] (x_0) d\xi  \Bigg) d\tau\\
        & = 2\pi \int_{\tau}e^{it\tau}|\tau|^{\frac{2s+1}{4}} \Bigg( \int_{x}  \widehat{f}^{(t)}(x_0-x,\tau) \widehat{g}^{(t)}(x,\tau)  dx  \Bigg) d\tau.\\
    \end{split}
\end{equation*}
Now notice, from Lemma \ref{2} that for $\tau\neq 0$ we have $\partial_x^s\widehat{g}(x_0,\tau) \eqsim |\tau|^{\frac{s}{2}}\widehat{g}(x_0,\tau)$
whenever $s\in \mathbb{N}$. So by integration by parts we have
\begin{equation*}
    III \eqsim \int_{\tau}e^{it\tau}|\tau|^{\frac{1}{4}} \Bigg( \int_{x}  \widehat{\partial_x^sf}^{(t)}(x_0-x,\tau) \widehat{g}^{(t)}(x,\tau)  dx  \Bigg).
\end{equation*}
As for $s\in \mathbb{Z}_-$, the $s$-th anti-derivative $\widehat{G}^{(t)}(x_0,\tau)$ of $\widehat{g}(x_0,\tau) $ satisfies
$$\widehat{g}(x_0,\tau) \eqsim \partial_x^{-s} \widehat{G}^{(t)}(x_0,\tau) = |\tau|^{-\frac{s}{2}} \widehat{G}^{(t)}(x_0,\tau).$$
So for $s\in \mathbb{Z}_-$ we can write $III$ as
\begin{equation}
    \begin{split}
        III & \eqsim  \int_{\tau}e^{it\tau}|\tau|^{\frac{1}{4}} \Bigg( \int_{x}  \widehat{f}^{(t)}(x_0-x,\tau) \widehat{G}^{(t)}(x,\tau)  dx  \Bigg)\\
        & = \int_{\tau}e^{it\tau}|\tau|^{\frac{1}{4}} \Bigg( \int_{x}  |\tau|^s\widehat{f}^{(t)}(x_0-x,\tau) \widehat{g}^{(t)}(x,\tau)  dx  \Bigg)\\
        & = \int_{\tau}e^{it\tau}|\tau|^{\frac{1}{4}} \Bigg( \int_{x}  \widehat{D^s_xf}^{(t)}(x_0-x,\tau) \widehat{g}^{(t)}(x,\tau)  dx  \Bigg). \label{bu}
    \end{split}
\end{equation}
It follows that, for $s\in \mathbb{Z}$, $III$ can be expressed as \eqref{bu}, where $D^s$ is in the sense of fractional derivative for negative $s$ and in the sense of usual derivative for nonnegative $s$. By Plancherel's theorem,
\begin{equation*}
    \begin{split}
        |III|^2_{L^2_t} \eqsim |\widehat{III}|^2_{L^2_{\tau}} & \lesssim \int_0^{\infty} \sqrt{\tau} \Bigg| \int_{x}  \widehat{D^s_xf}^{(t)}(x_0-x,\tau) \widehat{g}^{(t)}(x,\tau)  dx \Bigg|^2 d\tau \\
        & + \int_{-\infty}^0 \sqrt{\tau} \Bigg| \int_{x}  \widehat{D^s_xf}^{(t)}(x_0-x,\tau) \widehat{g}^{(t)}(x,\tau)  dx \Bigg|^2 d\tau\\ & = A_1+A_2.
    \end{split}
\end{equation*}
One can rewrite $A_1$ as
\begin{equation*}
\begin{split}
    A_1  \eqsim \int_{x,t,x',t'}K(x,t,x',t')D^s_xf(x_0-x,t)\overline{D_{x'}^sf(x_0-x',t')} dxdtdx'dt',
\end{split}
\end{equation*}
where
$K(x,t,x',t')=\int_0^{\infty} e^{-(|x|+|x'|)\xi} e^{i(t'-t)\xi^2} d\xi.$
A similar argument also applies to $A_2$. We have $K(x,t,x',t')\lesssim \frac{1}{\sqrt{|t-t'|}}$,
which implies through Hardy-Littlewood-Sobolev inequality that
$       |III|_{L_t^2}^2 \lesssim |f|_{L^{\frac{4}{3}}_t\dot{H}^{s,1}_x}^2.$ Now for fixed $s\in \mathbb{Z}$ we can interpolate between the pairs $(\lambda',r')=(1,2)$ and $(\frac{4}{3},1)$ to reach the desired estimate for $III$, which proves \eqref{homo_est} for the case $m=0$. For $m=1$, we can differentiate in $x$ and repeat the same arguments to obtain
$$|\partial_xz(x,\cdot)|_{\dot{H}_t^{\frac{2s-1}{4}}(0,T)}\lesssim  |\partial_x^mf|_{L_t^{\lambda'}\dot{H}_x^{s-1,r'}} \leq |f|_{L_t^{\lambda'}\dot{H}_x^{s,r'}}$$
for the same admissible pairs and $s\in \mathbb{Z}$. Finally we can interpolate over $(\lambda,r)$ for fixed $s\in \mathbb{Z}$ proving \eqref{homo_est}. If $s\in \mathbb{N}_0$ and $s<m$ (which is only possible if $s=0$ and $m=1$), then $\frac{2s+1-2m}{4}=-\frac{1}{4}<0$ and hence
\begin{equation*}
        |\partial_x^mz(x,\cdot)|_{H_t^{\frac{2s+1-2m}{4}}(0,T)}\leq |\partial_x^mz(x,\cdot)|_{\dot{H}_t^{\frac{2s+1-2m}{4}}(0,T)}\lesssim|f|_{L_t^{\lambda'}H^{s,r'}_x}.
\end{equation*}
Now suppose $s\in \mathbb{N}_0$ and $s\geq m$. Then $\frac{2s+1-2m}{4}\geq \frac{1}{4}$. Let $\theta_T$ be a smooth cut-off function such that $\theta_T|_{[0,T]}\equiv 1$. Then we have
\begin{equation*}
    \begin{split}
        |\partial_x^mz(x,\cdot)|_{H_t^{\frac{2s+1-2m}{4}}(0,T)} \leq & |\partial_x^mz(x,\cdot)|_{\dot{H}_t^{\frac{2s+1-2m}{4}}(0,T)}+ |\partial_x^mz(x,\cdot)|_{L^2}\\
        \lesssim & |f|_{L_t^{\lambda'}\dot{H}_x^{s,r}}+ |\theta_T\partial_x^mz(x,\cdot)|_{H_t^{\frac{1}{4}}(0,T)}\\
        \leq & |f|_{L_t^{\lambda'}\dot{H}_x^{s,r}}+c_T |\partial_x^mz(x,\cdot)|_{\dot{H}_t^{\frac{1}{4}}(0,T)}\\
        \lesssim & |f|_{L_t^{\lambda'}\dot{H}_x^{s,r}}+ c_T|f|_{L_t^{\lambda'}\dot{H}_x^{m,r}}\\
        \leq & (1+c_T)|f|_{L_t^{\lambda'}H_x^{s,r}}.
    \end{split}
\end{equation*}
Notice that the above argument need not hold when $T=\infty$.
We have proved \eqref{inhomo_est} for all $s\in \mathbb{N}_0$, so we can interpolate and extend the estimate to all $s\geq0$.
\end{proof}
\begin{rem}
The constant $c_T$ in \ref{inhomo_est} is an increasing function of $T$. If $T$ is assumed to be bounded by some $M>0$, one can omit the term $c_T$ by taking $\theta_M$ instead of $\theta_T$.
\end{rem}


\section{Nonlinear Applications}\label{Apps}
We will prove local well-posedness of the NLS and the coupled system of NLS equations on the halfline utilizing trace and Strichartz estimates given in previous sections. The local wellposedness of the Dirichlet and Neumann problems for the single NLS were studied before in the high regularity setting in \cite{FHM17} and \cite{Him19}, respectively, via Fokas method based formulas. Dirichlet problem was studied in the low regularity setting with other methods (see e.g., see \cite{Bona18}, \cite{Holmer}).  Therefore, here we will prove the local well-posedness of only the Neumann problem in the low regularity setting $s\in [0,\frac{1}{2})$ using Theorem \ref{mainthm2}.  Then we move on to the coupled system of NLS equations and establish its wellposedness, too.

One should not expect a local wellposesness result for negative indices in light of the theory of Cauchy problems.  Indeed, let $u$ be a solution of the NLS on the real line:
\begin{equation}\label{NLS}
  u_t+Pu+f(u)=0.
\end{equation}
Then, $u_{\epsilon}(x,t)\doteq \epsilon^{-\frac{2}{p}}u(\epsilon^{-1}x,\epsilon^{-2}t)$ is also a solution of the same eqution. Moreover, $u$ solves \eqref{NLS} on $(0,T)$  iff $u_\epsilon$ solves \eqref{NLS} on $(0,\epsilon^2 T)$. The corresponding initial data satisfy $\displaystyle|u_\epsilon(0)|_{\dot{H}_x^s}=\epsilon^{\frac{1}{2}-\frac{2}{p}-s}|u(0)|_{\dot{H}_x^s}.$ It is clear that if $s<s^*\doteq \frac{1}{2}-\frac{2}{p}$, then both $|u_\epsilon(0)|_{\dot{H}_x^s}$ and the life span of $u_\epsilon$ tends to zero as $\epsilon\rightarrow 0^+$. This suggests that the problem may be locally illposed for $s<s^*$ and locally wellposed otherwise.  Illposedness is easier to establish for $0\le s <s^*$ whenever $s^*> 0$ ($L^2$-\emph{supercritical}) or for $s <s^*=0$ ($L^2$-\emph{critical}). For the focusing problems, in the case $s^*\ge 0$, one can construct solutions that blow up in given arbitrarily small time.   However, if $p< 4$, then $s^*< 0$, in which case an explicit blow-up solution cannot be constructed. So, the question is what is the range of $s$ for which local wellposedness fails when $s^*<0$ ($L^2$-\emph{subcritical}).  It was shown in \cite{Tao} that the solution operator is not uniformly continuous for $s<\max(0,s^*)$.  This motivates us to consider the wellposedness problem for $s\ge \max(0,s^*)$ also in the present context of an ibvp.
\subsection{NLS with Neumann b.c.}
In this section we prove the local well posedness of the following Neumann problem for $0\leq s<\frac{1}{2}$.
\begin{subequations}\label{n_problem}
\begin{empheq}{align}
&iy_{t}+y_{xx}=\kappa|y|^py\eqqcolon f(y), \quad (x,t)\in\mathbb{R}_+\times (0,T),\\
&y(x,0)=y_0(x)\in H_x^s(\Real_+),\\
&y_x(0,t)=g(t)\in H_t^{\frac{2s-1}{4}}(0,T)
\end{empheq}
\end{subequations}
 where $\kappa\in \mathbb{C}$ and $p\leq \frac{4}{1-2s}$. Note that we have no compatibility conditions when $s\in [0,1/2)$ as traces are not defined.
\begin{thm}\label{neumannthm}
Let $T>0$, $s\in [0,\frac{1}{2})$, $0<p\leq \frac{4}{1-2s}$, $y_0\in H_x^s(\Real_+)$ and $g\in H_t^{\frac{2s-1}{4}}(0,T)$. Define Schrödinger admissable pair $\lambda = \frac{4(p+2)}{p(1-2s)},\  r = \frac{p+2}{1+sp}$. We assume $y_0$ is small if $p=4/(1-2s).$ Then \ref{n_problem} has the following local well-posedness properties:
\begin{itemize}
    \item[(i)] Local existence and uniqueness. There is a unique local solution $y\in L_t^{\lambda}(0,T_0;H_x^{s,r}(\Real_+))\cap C([0,T_0];H_x^s(\Real_+))$ for some $T_0\in (0,T]$.
    \item[(ii)] Continuous dependence. For any bounded subset $B$ of $H_x^s(\Real_+) \times H_t^{\frac{2s-1}{4}}(0,T)$ there is $T_0>0$ such that the map $(y_0,g)\in B \mapsto y\in C([0,T_0];H_x^s(\Real_+))$ is Lipschitz continuous.
    \item[(iii)] Blow-up alternative. Let $S$ be the set of $T_0\in (0,T]$ such that there is a unique local solution in $L^{\lambda}(0,T_0;H^{s,r}(\Real_+))$. Then
    $T_{max}\equiv \sup_{T_0\in S}T_0<T \Rightarrow \lim_{t\rightarrow T_{max}}|y(t)|_{H^s(\Real_+)}=\infty.$
\end{itemize}
\end{thm}
\begin{proof}
\textbf{Existence.}
A solution of \ref{n_problem} is a fixed point of the operator
\begin{align}\Theta[y](t)\coloneqq\, &\, v|_{\mathbb{R}_+\times (0,T_0)}+z|_{\mathbb{R}_+\times (0,T_0)}+u|_{(0,T_0)}\\
\equiv &\left.e^{-tP}y_0^*\right|_{\mathbb{R}_+\times (0,T_0)}+\left.\int_0^te^{-(t-t')P}{f^*}(y)dt'\right|_{\mathbb{R}_+\times (0,T_0)}+\left.\mathcal{T}_B(t)h\right|_{(0,T_0)},\end{align} where $y_0^*$ - spatially defined on $\mathbb{R}$ - is an extension of $y_0$ and similarly $f^*(y)$ is an extension of $f(y)$ (can be simply taken as $f(y^*)$ where $y^*$ is an extension of $y$ - this is justified later through the estimates of the nonlinear term).  Here, by abuse of notation we use $(\cdot)^*$ to denote all extensions. Moreover, these extensions correspond to fixed bounded regularity preserving extension operators. Therefore, in estimates one can switch between norms of $y_0, f$ and $y_0^*, f^*$, respectively.

Here, $h\in H_t^{\frac{2s-1}{4}}(\Real)$ with $\supp{h}\subset [0,T')$ for some $T'>T$ that is an extension (in the distributional sense for small $s$) of the boundary input $g-\partial_x v(0,\cdot)-\partial_x z(0,\cdot)$ which is defined on $[0,T]$ such that
$$|h|_{H_t^{\frac{2s-1}{4}}(\Real)}\lesssim |g-\partial_x v(0,\cdot)-\partial_x z(0,\cdot)|_{H_t^{\frac{2s-1}{4}}(0,T)}.$$
Indeed one can see that such $h$ exists by first taking an extension to $\Real$ and then multiplying this extension by a suitable cut-off function (see e.g. \cite{Agr} for construction of such extension).
Now consider the space $$X_{T_0}\coloneqq \{ y\in L_t^{\lambda}(0,T_0;H_x^{s,r}(\Real_+)): |y|_{L_t^{\lambda}(0,T_0;H_x^{s,r}(\Real_+))}\leq R \}$$
with the metric $d(y_1,y_2)=|y_1-y_2|_{L_t^{\lambda}(0,T_0;L_x^r(\Real_+))}.$
\begin{rem}
$(X_{T_0},d)$ is a complete metric space, letting us to use Banach's fixed point theorem. Although $X_{T_0}$ is not a linear space we will still write $|y|_{X_{T_0}}\coloneqq |y|_{L_t^{\lambda}(0,T_0;H_x^{s,r}(\Real_+))}$ to shorten the notation.
\end{rem}
We claim that there are $R,{T_0}>0$ such that $\Theta:X_{T_0}\rightarrow X_{T_0}$ is a contraction. We will first show that $\Theta(X_{T_0})\subseteq X_{T_0}$ for some $R,{T_0}>0$. Suppose $y\in X_{T_0}$. Then, by the linear theory
\begin{align*}
    |\Theta y|_{X_{T_0}}     \lesssim & |y_0|_{H^s_x(\mathbb{R}_+)}+ |f(y^*)|_{L_t^{\lambda'}(0,T;H_x^{s,r'}(\Real))}+|u|_{X_{T_0}}.
\end{align*}
\par Again by the linear theory,
\begin{align*}
    |u|_{X_{T_0}} &\lesssim |h|_{H_t^{\frac{2s-1}{4}}(\Real)}\leq |g-\partial_x v(0,\cdot)-\partial_x z(0,\cdot)|_{H_t^{\frac{2s-1}{4}}(0,T_0)} \\
    &\leq |g|_{H_t^{\frac{2s-1}{4}}(0,{T_0})}+|y_0|_{H^s_x(\mathbb{R}_+)}+|f(y^*)|_{L_t^{\lambda'}(0,{T_0};H_x^{s,r'}(\Real))}.
\end{align*}
So we have
\begin{equation}
    |\Theta y|_{X_{T_0}}\leq c_{T_0} (|y_0|_{H^s_x(\mathbb{R}_+)}+|f(y^*)|_{L_t^{\lambda'}(0,{T_0};H_x^{s,r'}(\Real_+))}+|g|_{H_t^{\frac{2s-1}{4}}(0,{T_0})})
\end{equation}
where $c_{T_0}$ is a constant dependent on ${T_0}$. Now take $R=2c_{T_0}|y_0|_{H^s_x(\mathbb{R}_+)}$. Let $(\lambda,r)$ be the particular admissible pair
$\lambda \coloneqq \frac{4(p+2)}{p(1-2s)}, r \coloneqq \frac{p+2}{1+sp}.$
Then by the fractional Leibniz and chain rules:
\begin{equation*}
    |f(y^*)|_{L_t^{\lambda'}(0,{T_0};\dot{H}_x^{s,r'}(\Real))}\leq T_0^{\theta}|y^*|^{p+1}_{L_t^{\lambda}(0,{T_0};\dot{H}_x^{s,r}(\Real))}\leq T_0^{\theta}|y|^{p+1}_{X_{T_0}},
\end{equation*}
\begin{equation*}
    |f(y^*)|_{L_t^{\lambda'}(0,{T_0};L^{r'}(\Real))}\leq T_0^{\theta} |y^*|^{p}_{L_t^{\lambda}(0,{T_0};\dot{H}_x^{s,r}(\Real))}|y^*|_{L_t^{\lambda}(0,{T_0};L_x^{r}(\Real))}\leq T_0^{\theta}|y|^{p+1}_{X_{T_0}}
\end{equation*}
where $\theta \coloneqq 1-\frac{p(1-2s)}{4}$. We have
\begin{equation*}
    |\Theta y|_{X_{T_0}} \leq \frac{R}{2}+ c_{T_0}\big(T_0^{\theta}R^{p+1}+|g|_{H_t^{\frac{2s-1}{4}}(0,{T_0})}\big).
\end{equation*}
From the linear theory, $c_{T_0}$ is non-increasing as ${T_0}$ gets smaller. Thus we can take ${T_0}$ small such that
\begin{equation*}
    c_{T_0}\big(T_0^{\theta}R^{p+1}+|g|_{H_t^{\frac{2s-1}{4}}(0,{T_0})}\big)\leq \frac{R}{2}
\end{equation*}
so that we have $\Theta y\in X_{T_0}$. Now we have to show that $\Theta$ is a contraction. Take $y_1,y_2\in X_{T_0}$. Then
\begin{align}
    d(\Theta y_1,\Theta y_2)& \lesssim  |f(y_1^*)-f(y_2^*)|_{L_t^{\lambda'}(0,{T_0};H_x^{s,r'}(\Real))}\nonumber \\
    & \lesssim T_0^{\theta}\Big(|y_1^*|^p_{L_t^{\lambda}(0,{T_0};\dot{H}_x^{s,r}(\Real))} + |y_2^*|^p_{L_t^{\lambda}(0,{T_0};\dot{H}_x^{s,r}(\Real))}\Big)  |y_1^*-y_2^*|_{L_t^{\lambda}(0,{T_0};L_x^r(\Real))}\nonumber \\
    & \leq T_0^{\theta}\big(|y_1|^p_{X_{T_0}}+|y_2|^p_{X_{T_0}}\big)d(y_1,y_2)\label{uniqueness_arg}.
\end{align}
So taking ${T_0}$ small enough we can make $\Theta$ a contraction. Then by Banach fixed point theorem $\Theta$ has a unique fixed point $y\in X_{T_0}$. From the linear theory, it also follows that $y\in C([0,{T_0}];H^{s}_x(\Real_+))$.
\begin{rem}
Note that the above arguments apply in the subcritical case $\theta>0$. For the critical case $\theta=0$ the term $T_0^{\theta}$ on the RHS is lost but one can still argue in a similar manner with the additional assumption that initial data is small.
\end{rem}
\textbf{Uniqueness.}
Let $y_1,y_2\in L^{\lambda}_t([0,{T_0}],H_x^{s,r}(\Real_+))$ be two solutions. Then by \eqref{uniqueness_arg} we have
\begin{equation*}
    d(y_1,y_2)=d(\Theta y_1,\Theta y_2)\lesssim T_0^{\theta}(|y_1|^p_{X_{T_0}}+|y_2|^p_{X_{T_0}})d(y_1,y_2)
\end{equation*}
Hence $T_0$ can be taken small enough to guarantee uniqueness for $\theta>0$. For the critical case, a similar argument applies for small initial data.
\\
\textbf{Continuous Dependence.} Let $B$ be bounded subset of $H_x^s(\Real_+)\times H_t^{\frac{2s-1}{4}}(0,T_0)$ and let $(y_1,g_1),(y_2,g_2)\in B$ with corresponding solutions $w_1$ and $w_2$ in $X_{T_0}$, respectively. Then $w\equiv w_1-w_2$ is a solution to the following problem
\begin{subequations}
\begin{empheq}{align}
&iw_{t}+w_{xx}=f(w_1)-f(w_2), \quad (x,t)\in\mathbb{R}_+\times (0,T_0),\\
&w(x,0)= (y_1-y_2)(x)\in H_x^s(\Real_+),\\
&w_x(0,t)= (g_1-g_2)(t)\in H_t^{\frac{2s-1}{4}}(0,T_0).
\end{empheq}
\end{subequations}
Then we have
\begin{align*}
    |w|_{X_{T_0}} \lesssim &|y_1^*-y_2^*|_{H_x^s(\Real_+)} + |f(w_1^*)-f(w_2^*)|_{L_t^{\lambda'}(0,{T_0};H_x^{s,r'}(\Real))} + c_{T_0}|h_1-h_2|_{H_t^{\frac{2s-1}{4}}(\Real)}\\
     \lesssim &|y_1-y_2|_{H^s(\Real_+)} + T_0^{\theta}\Big(|w_1|^p_{X_{T_0}}+{|w_2|}^p_{X_{T_0}}\Big)|w_1-w_2|_{X_{T_0}}+ c_{T_0}|g_1-g_2|_{H_t^{\frac{2s-1}{4}}(0,{T_0})}.
\end{align*}
Then taking ${T_0}$ small enough (provided $\theta>0$) we obtain
$
    |w_1-w_2|_{Y_{T_0}}\lesssim |(y_1,g_1)-(y_2,g_2)|_{B}.
$A similar argument also applies to $\theta=0$ case by assuming initial data are small. \\
\textbf{Blowup Alternative.} Let $S$ be the set of $T_0\in (0,T]$ such that there is a unique local solution in $L_t^{\lambda}(0,T_0;H_x^{s,r}(\Real_+))$. We claim that if $T_{max}\coloneqq \sup_{T_0\in S}T_0<T$ then $|y(t)|_{H_x^s(\Real_+)}\rightarrow \infty$ as $t\rightarrow T_{max}$. Suppose for contradiction that there is a sequence $(t_n)\subset S$ such that $t_n\rightarrow T$ and $|y(t_n)|_{H_x^s(\Real_+)}\leq M$ for some $M>0$. Given any $n\in \mathbb{N}$ there is a unique solution $y_1$ of the problem on $[0,t_n]$. Now consider the following problem
\begin{subequations}\label{blowup}
\begin{empheq}{align}
&iy_{t}+y_{xx}=\kappa|y|^py \quad (x,t)\in\mathbb{R}_+\times (t_n,T),\\
&y(x,t_n)=y_0(x),\\
&y_x(0,t)=g(t)
\end{empheq}
\end{subequations}
Then \eqref{blowup} has a unique local solution $y_2$ on some interval $[t_n,t_n+\delta]$. Now choose $n$ such that $t_n+\delta>T_{max}$. Define $y$ such that $y\equiv y_1$ on $[0,t_n)$ and $y\equiv y_2$ on $[t_n,t_n+\delta]$. Then $y$ is a local solution on $[0,t_n+\delta]$ which is a contradiction.
\end{proof}
\subsection{Coupled system of NLS equations}
In this section we prove local well-posedness for the cubic coupled NLS ibvp on the half line. We consider the following system:
\begin{subequations}\label{couple}
\begin{empheq}{align}
ip_t+p_{xx} & = a|q|^2p, & (x,t)\in \Real_+ \times (0,T) \label{coupled_eq_1},\\
p(x,0) & =p_0(x),\\
\gamma_mp(0,t) & =g(t),\\
iq_t+q_{xx} & =b|p|^2q, & (x,t)\in \Real_+ \times (0,T) \label{coupled_eq_2} \\
q(x,0) & =q_0(x),\\
\gamma_nq(0,t) & =h(t),
\end{empheq}
\end{subequations}
where $a,b\in \mathbb{C}$ and $m,n\in \{0,1\}$.
\begin{thm}[Low regularity solutions]\label{couple_thm_low}
Let $s\in [0,1/2)$, $\lambda = \frac{8}{1-2s}$ and $r=\frac{4}{1+2s}$. Let $p_0,q_0\in H_x^s(\Real_+)$ and for $m,n\in \{0,1\}$ let $g\in H_t^{\frac{2s+1-2m}{4}}(0,T)$ and $h\in H_t^{\frac{2s+1-2n}{4}}(0,T)$. Then there is $T_0\in(0,T]$ such that the system \eqref{couple} has a unique solution $$(p,q)\in L_t^{\lambda}(0,T_0;H_x^{s,r}(\Real_+))^2\cap C([0,T_0];H_x^s(\Real_+))^2$$ Moreover the data-to-solution map is locally Lipschitz continuous.
\end{thm}
\begin{proof}
We will again use a fixed point argument. Let $\Psi_m[y_0,g,f]$ denote the solution operator of \ref{lineq} for $B=\gamma_m$.
Then it is easy to see that a solution of \eqref{couple}
corresponds to a fixed point of the operator
\begin{equation}\label{lambda}
    \Lambda : (p,q)\mapsto \Big(\Psi_m[p_0,g,a|q|^2p],\Psi_n[q_0,h,b|p|^2q]\Big)
\end{equation}
We define the space $X_{T_0}$ for the sought-after solution as
\begin{equation*}
    X_{T_0}\coloneqq \{(p,q)\in L_t^{\lambda}(0,T_0;H_x^{s,r}(\Real_+))^2: |p|_{L_t^{\lambda}(0,T_0;H_x^{s,r}(\Real_+))}+|q|_{L_t^{\lambda}(0,T_0;H_x^{s,r}(\Real_+))}\leq R \}
\end{equation*}
equipped with the metric
\begin{equation*}
    d((p_1,q_1),(p_2,q_2))= |p_1-p_2|_{L_t^{\lambda}(0,T_0;L_x^{r}(\Real_+))}+|q_1-q_2|_{L_t^{\lambda}(0,T_0;L_x^{r}(\Real_+))}.
\end{equation*}
So that $(X_{T_0},d)$ is a complete metric space. We claim that there is $T_0>0$ such that $\Lambda$ has a fixed point in $X_{T_0}$.
We have, similar to the proof of Theorem \ref{neumannthm},
\begin{align*}
    |\Lambda (p,q)|_{X_{T_0}} \leq  & c_{T_0} \big(|p_0|_{H^s_x(\mathbb{R}_+)}+|a|q|^2p|_{L_t^{\lambda'}(0,{T_0};H_x^{s,r'}(\Real_+))}+|g|_{H_t^{\frac{2s+1-2m}{4}}(0,{T_0})}\\
    & +|q_0|_{H^s_x(\mathbb{R}_+)}+|b|p|^2q|_{L_t^{\lambda'}(0,{T_0};H_x^{s,r'}(\Real_+))}+|h|_{H_t^{\frac{2s+1-2n}{4}}(0,{T_0})}\big).
\end{align*}
Thus, for the invariance of $X_{T_0}$, it will suffice to show the following estimate
\begin{equation}\label{est_1}
    ||p|^2q|_{L^{\lambda'}(0,T_0;H^{s,r'}(\Real_+)}+||q|^2p|_{L^{\lambda'}(0,T_0;H^{s,r'}(\Real_+))}\lesssim c(T_0,R,{\theta})
\end{equation} with $c(T_0,R,{\theta})$ getting smaller as $T_0$ gets smaller. To prove that $\Lambda$ is a contraction it is enough to show that for $(p_1,q_1)$ and $(p_2,q_2)$ in $X_{T_0}$, we have
\begin{align}\label{est_2}
    & ||q_1|^2p_1-|q_2|^2p_2|_{ L_t^{\lambda'}(0,T_0;L_x^{r'}(\Real_+))}+|  |p_1|^2q_1-|p_2|^2q_2|_{L_t^{\lambda'}(0,T_0;L_x^{r'}(\Real_+))}\nonumber
    \\ \lesssim & c(T_0,R,{\theta}) |(p_1,q_1)-(p_2,q_2)|_{L_t^{\lambda}(0,T_0;L_x^{r}(\Real_+))}.
\end{align}

We choose $\mu$ and $\rho$ such that
$    \frac{1}{\lambda'}=\frac{1}{\mu}+\frac{2}{\lambda},$
   $ \frac{1}{r'}=\frac{2}{\rho}+\frac{1}{r} $.
We have the Sobolev embedding $\dot{H}^{s,r'}\xhookrightarrow{} L^{\rho}$ and $\mu < \lambda$. Let
\begin{equation}\label{def_theta}
    \theta \coloneqq \frac{1}{\mu}-\frac{1}{\lambda} > 0
\end{equation}
and let
\begin{equation}\label{def_j}
    \frac{1}{j}\coloneqq \frac{1}{\rho}+\frac{1}{r}.
\end{equation}
Then given $(p,q)\in X_{T_0}$, by the generalized Hölder inequality, we have
\begin{align*}
    |D^s(|p^*(t)|^2q^*(t))|_{{L}_x^{r'}} & \lesssim |D^s(|p^*(t)|^2)|_{L_x^j} |q^*(t)|_{L_x^{\rho}} + ||p^*(t)|^2|_{L_x^{\frac{\rho}{2}}} |D^sq^*(t)|_{L_x^r}.
\end{align*}
The first term on the RHS is estimated as
\begin{equation}\label{iterate}
    |D^s(|p^*(t)|^2)|_{L_x^j}\lesssim |D^sp^*(t)|_{L_x^r}|\overline{p}^*(t)|_{L_x^{\rho}} \leq |p^*(t)|^2_{\dot{H}_x^{s,r}}.
\end{equation}
Thus we have
\begin{equation}\label{estt}
    |D^s(|p^*|^2q^*)|_{L_x^{r'}} \lesssim |p^*|_{\dot{H}_x^{s,r}}^2|q^*|_{\dot{H}_x^{s,r}} \lesssim |p|_{\dot{H}_x^{s,r}(\Real_+)}^2|q|_{\dot{H}_x^{s,r}(\Real_+)}.
\end{equation}
Now applying Hölder's inequality in time, using \eqref{estt} and \eqref{def_theta}, we get
\begin{align*}
    ||p|^2q|_{L_t^{\lambda'}(0,T;\dot{H}_x^{s,r'}(\Real_+))} & \leq ||p^*|^2q^*|_{L_t^{\lambda'}(0,T;\dot{H}_x^{s,r'})}\\
    & \leq \big ||p|_{\dot{H}_x^{s,r}(\Real_+)}^2 \big |_{L_t^{\frac{\lambda}{2}}(0,T_0)} \big ||q|_{\dot{H}_x^{s,r}(\Real_+)} \big |_{L_t^{\mu}(0,T_0)}\\
    & \leq T_0^{\theta} |p|^2_{L_t^{\lambda}(0,T_0;\dot{H}_x^{s,r}(\Real_+)} |q|_{L_t^{\lambda}(0,T_0;\dot{H}_x^{s,r}(\Real_+))} \leq T_0^{\theta}R^3.
\end{align*}
Using the same $\mu$ and $\rho$, applying Hölder inequality twice gives
\begin{equation*}
    ||p|^2q|_{L^{\lambda'}(0,T;L^{r'}(\Real_+))} \leq T_0^\theta |p|^2_{L^{\lambda}(0,T_0;\dot{H}^{s,r}(\Real_+))} |q|_{L^{\lambda}(0,T_0;L^{r}(\Real_+))}\leq c_2T_0^{\theta}R^3.
\end{equation*}
Thus we can choose $T_0$ and $R$ to enforce invariance of $X_{T_0}$ under $\Lambda$. Now observe that we have
\begin{align}\label{mnp}
|p_1|^2q_1-|p_2|^2q_2 & =\frac{1}{2}(q_1-q_2) (|p_1|^2+|p_2|^2)\nonumber\\
& +\frac{1}{2}(q_1+q_2)\big[(p_1-p_2)(\overline{p_1}+\overline{p_2})+(p_1+p_2)(\overline{p_1}-\overline{p_2})\big]
\end{align}
from which we obtain
\begin{align*}
    ||p_1^*|^2q_1^*-|p_2^*|^2q_2^*|_{ L_t^{\lambda'}(0,T_0;L_x^{r'})} & \lesssim |q_1-q_2|_{L_t^{\mu}(0,T_0;L_x^r)}||p_1|^2+|p_2|^2|_{L_t^{\frac{\lambda}{2}}(0,T_0;L_x^{\frac{\rho}{2}})}\\ & + |p_1-p_2|_{L_t^{\mu}(0,T_0;L_x^r)}|(\overline{p_1}+\overline{p_2})(q_1+q_2)|_{L_t^{\frac{\lambda}{2}}(0,T_0;L_x^{\frac{\rho}{2}})}\\
    & + |\overline{p_1}-\overline{p_2}|_{L_t^{\mu}(0,T_0;L_x^r)}|(p_1+p_2)(q_1+q_2)|_{L_t^{\frac{\lambda}{2}}(0,T_0;L_x^{\frac{\rho}{2}})}\\
    & \lesssim T^{\theta}_0 R^2 d((p_1,q_1),(p_2,q_2)).
\end{align*}
So we have proved \eqref{est_1} and \eqref{est_2}. The uniqueness of the fixed point and the local Lipschitz continuity of the data-to-solution map follows as in the proof of Theorem \ref{neumannthm}.
\end{proof}
Now we will prove local well-posedness for high regularity solutions of \eqref{couple}. Throughout, we will assume the necessary compatibility conditions between initial and boundary data.
\begin{thm}[High regularity solutions]\label{couple_thm_high}
Let $s\in (\frac{1}{2}, \frac{5}{2})-\frac{3}{2}$, $p_0,q_0\in H_x^s(\Real_+)$ and for $m,n\in \{0,1\}$ let $g\in H_t^{\frac{2s+1-2m}{4}}(0,T)$ and $h\in H_t^{\frac{2s+1-2n}{4}}(0,T)$ satisfying necessary compatibility conditions. Then there is $T_0\in (0,T]$ such that the system \eqref{couple} has a unique solution $$(p,q)\in C([0,T_0];H_x^s(\Real_+))^2$$ and the data-to-solution map is locally Lipschitz continuous.
\end{thm}
\begin{proof}
We have to prove that the operator $\Lambda$ in \eqref{lambda} is continuous from $Y_{T_0}$ to itself for some ${T_0},R>0$ where
$$Y_{T_0}\coloneqq \{ (p,q)\in C([0,T_0];H^s(\Real_+))^2: |p|_{C([0,T_0];H^s(\Real_+))}+|q|_{C([0,T_0];H^s(\Real_+))}\leq R \}$$ equipped with the distance
$$d((p_1,q_1),(p_2,q_2))=|p_1-p_2|_{C([0,T_0];H^s(\Real_+))}+|q_1-q_2|_{C([0,T_0];H^s(\Real_+))}$$
From the proof of local well-posedness for the Neumann and Dirichlet problems for the NLS in the high regularity case (see \cite{FHM17} for the Dirichlet problem) one can see that it will suffice to deal only with the nonlinear terms $a|q|^2p$ and $b|p|^2q$. This case will be much simpler than the low regularity case because we can make use of the Banach algebra property. First, we have the following proposition
\begin{prop}[\cite{FHM17}]\label{f_prop}
    Let $f\in C([0,T];H^s(\Real_+))$, then
    $$\sup_{t\in [0,T]}|\Psi[0,0,f]|_{C([0,T];H^s(\Real_+))} \lesssim T \sup_{t\in [0,T]}|f|_{C([0,T];H^s(\Real_+))}.$$
\end{prop}
It follows from Proposition \ref{f_prop} and the submultiplicativity of the norm that
$$|\Psi[0,0,a|p|^2q]|_{Y_{T_0}}\leq c T_0 (|p|_{Y_{T_0}}^2|q|_{Y_{T_0}}+|q|_{Y_{T_0}}^2|p|_{Y_{T_0}})\leq c T_0 R^3,$$
where $c$ does not depend on $T_0$ or $R$. The same argument can be made for $|\Psi[0,0,b|q|^2p]|_{Y_{T_0}}$. Thus $T_0$ and $R$ can be taken small enough to make $Y_{T_0}$ invariant under $\Lambda$.

It follows trivially from \eqref{mnp} and the algebra property that $\Lambda$ can also be a contraction, exhibiting a fixed point $(p,q)\in Y_{T_0}$. Uniqueness and continuous dependece of this local solution can be proved in the same way as in the proof of Theorem \ref{neumannthm}.
\end{proof}
\begin{rem}
    The results in Theorems \ref{couple_thm_low} and \ref{couple_thm_high} can be extended to any positive integer power non-linearity by taking $a|q|^kp$ and $b|p|^kq$ where $k\in \mathbb{N}$ as the nonlinear term in \ref{coupled_eq_1} and \ref{coupled_eq_2}, respectively. In this case we let $\lambda$ and $r$ be as in Theorem \ref{neumannthm}, and take $\mu$ and $\rho$ such that
    $
    \frac{1}{\lambda'}=\frac{1}{\mu}+\frac{k}{\lambda},$
    $\frac{1}{r'}=\frac{k}{\rho}+\frac{1}{r}.$ Then we can go over the same arguments, iterating \eqref{iterate} $k$ times to obtain the desired result.
\end{rem}
\begin{rem}[Global wellposedness with inhomogeneous boundary data]
    Global well-posedness for the coupled system is a more difficult problem than it is for the single NLS due to the term $a|q|^2p\overline{p_t}$ that arises upon multiplying equation \eqref{coupled_eq_1} by $\overline{p_t}$. It is not easy to control this term, however, it is possible to prove global wellposedness for sufficiently small power indices. Consider for example the problem
    \begin{subequations}\label{small_power}
        \begin{empheq}{align}
            ip_t+p_{xx} & = a|q|^{\alpha}p & (x,t)\in \Real_+ \times (0,T) \label{small_p}\\
            p(x,0) & =p_0(x)\\
            \partial_xp(0,t) & =g(t)\\
            iq_t+q_{xx} & =b|p|^{\alpha}q & (x,t)\in \Real_+ \times (0,T) \label{small_q} \\
            q(x,0) & =q_0(x)\\
            \partial_xq(0,t) & =h(t)
        \end{empheq}
    \end{subequations}
     with $a,b\in \mathbb{R}$. Assuming a local solution ($H_x^1(\mathbb{R}_+)$ in space) exists, one can prove global wellposedness by showing that $H^1$ norm of the solution, that is $|p(t)|_{H^1_x(\Real_+)}+|q(t)|_{H^1_x(\Real_+)}$ does not blow up in finite time. We claim that it is controlled by the sum of $H^1$ norms of initial and boundary inputs uniformly in $t\in [0,T]$.

     Upon multiplying \eqref{small_p} with $\overline{p}$, taking imaginary parts, integrating in space-time, using Sobolev trace inequality and Cauchy-Schwarz inequality one obtains:
     \begin{equation}\label{l2_estp}
         |p(t)|^2_{L^2_x(\Real_+)}\lesssim |p_0|^2_{L^2_x(\Real_+)} + |g|_{L_t^{2}(0,T)}\left(\int_0^t|p(s)|_{H^1_x(\Real_+)}^2ds\right)^{1/2}.
     \end{equation}
 Similar arguments applied to the second equation yield
      \begin{equation}\label{l2_estq}
 	|q(t)|^2_{L^2_x(\Real_+)}\lesssim |q_0|^2_{L^2_x(\Real_+)} + |h|_{L_t^{2}(0,T)}\left(\int_0^t|q(s)|_{H^1_x(\Real_+)}^2ds\right)^{1/2}.
 \end{equation}
Note that same multipliers also lead to the estimates $$\int_0^\infty\partial_t|p(t)|^2dx\lesssim |g|_{H_t^{1}(0,T)}|p(t)|_{H_x^1(\mathbb{R}_+)}$$ and $$\int_0^\infty\partial_t|q(t)|^2dx\lesssim |h|_{H_t^{1}(0,T)}|q(t)|_{H_x^1(\mathbb{R}_+)}.$$
     Multiplying \eqref{small_p} by $\overline{p}_t$, taking real parts, integrating in space-time,  using integration by parts, Sobolev embedding $H_t^1(0,T)\hookrightarrow L_t^\infty(0,T)$, Sobolev trace theorem and making use of the above inequalities we get
     \begin{align}\label{rel_2}
         |p_x(t)|_{L_x^2(\mathbb{R}_+)}^2 \lesssim & |p_0'|^2_{L^2_x(\Real_+)} + |g(0)|\cdot |p_0(0)|+|g|_{H_t^1(0,T)}|p(t)|_{H_x^1(\mathbb{R}_+)}\nonumber\\
         &+|g|_{H_t^1(0,T)}\left(\int_0^t|p(s)|_{H_x^1(\mathbb{R}_+)}^2ds\right)^{1/2}\nonumber\\
         &+|a|\cdot|g|_{H_t^1(0,T)}\cdot\int_0^t|q(s)|_{H_x^1(\mathbb{R}_+)}^\alpha |p(s)|_{H_x^1(\mathbb{R}_+)}ds.
     \end{align}
A similar estimate also holds for the second equation:
     \begin{align}\label{newq00}
	|q_x(t)|_{L_x^2(\mathbb{R}_+)}^2 \lesssim & |q_0'|^2_{L^2_x(\Real_+)} + |h(0)|\cdot |q_0(0)|+|h|_{H_t^1(0,T)}|q(t)|_{H_x^1(\mathbb{R}_+)}\nonumber\\
	&+|h|_{H_t^1(0,T)}\left(\int_0^t|q(s)|_{H_x^1(\mathbb{R}_+)}^2ds\right)^{1/2}\nonumber\\
	&+|b|\cdot|h|_{H_t^1(0,T)}\cdot\int_0^t|p(s)|_{H_x^1(\mathbb{R}_+)}^\alpha |q(s)|_{H_x^1(\mathbb{R}_+)}ds.
\end{align}
Set $E(t):= |p(t)|_{H^1_x(\Real_+)}^2+|q(t)|_{H^1_x(\Real_+)}^2$. Then, in view of above estimates we get an inequality of the form
$$E(t)\lesssim c_1+c_2\int_0^tE(s)ds+c_3\int_0^tE^{\frac{1+\alpha}{2}}(s)ds,$$ with positive constants $c_i$, $i=1,2,3,$ depending on fixed parameters such as $a$, $b$, $T$, $|p_0|_{H^1_x(\mathbb{R}_+)}$, $|q_0|_{H^1_x(\mathbb{R}_+)}$, $|g|_{H_t^1(0,T)}$, and $|h|_{H_t^1(0,T)}$.  If $\alpha\le 1$, then this inequality reduces to
$$E(t)\lesssim C_1+C_2\int_0^tE(s)ds,$$ from which we obtain the desired uniform bound using Gronwall's inequality. Hence, the global wellposedness follows.  Problem remains open for large $\alpha$.
\end{rem}

\bibliographystyle{acm}
\bibliography{impbib}

\begin{thebibliography}{10}

\bibitem{Agr}
{\sc Agranovich, M.~S.}
\newblock {\em Sobolev Spaces, Their Generalizations and Elliptic Problems in
  Smooth and Lipschitz Domains}, vol.~- of {\em Springer Monographs in
  Mathematics}.
\newblock -, -.

\bibitem{Aud19}
{\sc Audiard, C.}
\newblock Global {S}trichartz estimates for the {S}chr\"{o}dinger equation with
  non zero boundary conditions and applications.
\newblock {\em Ann. Inst. Fourier (Grenoble) 69}, 1 (2019), 31--80.

\bibitem{Batal16}
{\sc Batal, A., and \"Ozsar\i, T.}
\newblock Nonlinear {S}chr\"odinger equations on the half-line with nonlinear
  boundary conditions.
\newblock {\em Electron. J. Differential Equations\/} (2016), Paper No. 222,
  20.

\bibitem{Ben-Artzi}
{\sc Ben-Artzi, M., and Tréves, F.}
\newblock Uniform estimates for a class of evolution equations.
\newblock {\em J. Funct. Anal. 120\/} (1994), 264--299.

\bibitem{Bona18}
{\sc Bona, J.~L., Sun, S.-M., and Zhang, B.-Y.}
\newblock Nonhomogeneous boundary-value problems for one-dimensional nonlinear
  {S}chr\"{o}dinger equations.
\newblock {\em J. Math. Pures Appl. (9) 109\/} (2018), 1--66.

\bibitem{Caz03}
{\sc Cazenave, T.}
\newblock {\em Semilinear {S}chr\"{o}dinger equations}, vol.~10 of {\em Courant
  Lecture Notes in Mathematics}.
\newblock New York University, Courant Institute of Mathematical Sciences, New
  York; American Mathematical Society, Providence, RI, 2003.

\bibitem{Tao}
{\sc Christ, M., Colliander, J., and Tao, T.}
\newblock Ill-posedness for nonlinear schr{\" o}dinger and wave equations.
\newblock {\em arXiv:math/0311048 [math.AP]\/}.

\bibitem{col02}
{\sc Colliander, J.~E., and Kenig, C.~E.}
\newblock The generalized {K}orteweg-de {V}ries equation on the half line.
\newblock {\em Comm. Partial Differential Equations 27}, 11-12 (2002),
  2187--2266.

\bibitem{Esq19}
{\sc Esquivel, L., Hayashi, N., and Kaikina, E.~I.}
\newblock Inhomogeneous {D}irichlet-boundary value problem for one dimensional
  nonlinear {S}chr\"{o}dinger equations via factorization techniques.
\newblock {\em J. Differential Equations 266}, 2-3 (2019), 1121--1152.

\bibitem{Fbook}
{\sc Fokas, A.~S.}
\newblock {\em A unified approach to boundary value problems}, vol.~78 of {\em
  CBMS-NSF Regional Conference Series in Applied Mathematics}.
\newblock Society for Industrial and Applied Mathematics (SIAM), Philadelphia,
  PA, 2008.

\bibitem{FHM17}
{\sc Fokas, A.~S., Himonas, A.~A., and Mantzavinos, D.}
\newblock The nonlinear {S}chr\"{o}dinger equation on the half-line.
\newblock {\em Trans. Amer. Math. Soc. 369}, 1 (2017), 681--709.

\bibitem{Hay21b}
{\sc Hayashi, N., Kaikina, E.~I., and Ogawa, T.}
\newblock Inhomogeneous {D}irichlet boundary value problem for nonlinear
  {S}chr\"{o}dinger equations in the upper half-space.
\newblock {\em Partial Differ. Equ. Appl. 2}, 6 (2021), Paper No. 69, 24.

\bibitem{Hay21}
{\sc Hayashi, N., Kaikina, E.~I., and Ogawa, T.}
\newblock Inhomogeneous {N}eumann-boundary value problem for nonlinear
  {S}chrodinger equations in the upper half-space.
\newblock {\em Differential Integral Equations 34}, 11-12 (2021), 641--674.

\bibitem{Him20}
{\sc Himonas, A.~A., and Mantzavinos, D.}
\newblock Well-posedness of the nonlinear {S}chr\"{o}dinger equation on the
  half-plane.
\newblock {\em Nonlinearity 33}, 10 (2020), 5567--5609.

\bibitem{Him19}
{\sc Himonas, A.~A., Mantzavinos, D., and Yan, F.}
\newblock The nonlinear {S}chr\"{o}dinger equation on the half-line with
  {N}eumann boundary conditions.
\newblock {\em Appl. Numer. Math. 141\/} (2019), 2--18.

\bibitem{Holmer}
{\sc Holmer, J.}
\newblock The initial-boundary-value problem for the 1{D} nonlinear
  {S}chr\"{o}dinger equation on the half-line.
\newblock {\em Differential Integral Equations 18}, 6 (2005), 647--668.

\bibitem{HolmerThesis}
{\sc Holmer, J.~A.}
\newblock {\em Uniform estimates for the {Z}akharov system and the
  initial-boundary value problem for the {K}orteweg-de {V}ries and nonlinear
  {S}chroedinger equations}.
\newblock ProQuest LLC, Ann Arbor, MI, 2004.
\newblock Thesis (Ph.D.)--The University of Chicago.

\bibitem{Ken91}
{\sc Kenig, C.~E., Ponce, G., and Vega, L.}
\newblock Oscillatory integrals and regularity of dispersive equations.
\newblock {\em Indiana Univ. Math. J. 40}, 1 (1991), 33--69.

\bibitem{Ken93}
{\sc Kenig, C.~E., Ponce, G., and Vega, L.}
\newblock Well-posedness and scattering results for the generalized
  {K}orteweg-de {V}ries equation via the contraction principle.
\newblock {\em Comm. Pure Appl. Math. 46}, 4 (1993), 527--620.

\bibitem{Ozs12}
{\sc \"{O}zsar\i, T.}
\newblock Weakly-damped focusing nonlinear {S}chr\"{o}dinger equations with
  {D}irichlet control.
\newblock {\em J. Math. Anal. Appl. 389}, 1 (2012), 84--97.

\bibitem{Ozs18}
{\sc \"{O}zsar\i, T.}
\newblock Global existence and open loop exponential stabilization of weak
  solutions for nonlinear {S}chr\"{o}dinger equations with localized external
  {N}eumann manipulation.
\newblock {\em Nonlinear Anal. 80\/} (2013), 179--193.

\bibitem{Ozs15}
{\sc \"{O}zsar\i, T.}
\newblock Well-posedness for nonlinear {S}chr\"{o}dinger equations with
  boundary forces in low dimensions by {S}trichartz estimates.
\newblock {\em J. Math. Anal. Appl. 424}, 1 (2015), 487--508.

\bibitem{Ozs11}
{\sc \"{O}zsar\i, T., Kalantarov, V.~K., and Lasiecka, I.}
\newblock Uniform decay rates for the energy of weakly damped defocusing
  semilinear {S}chr\"{o}dinger equations with inhomogeneous {D}irichlet
  boundary control.
\newblock {\em J. Differential Equations 251}, 7 (2011), 1841--1863.

\bibitem{OY19}
{\sc \"{O}zsar{\i}, T., and Yolcu, N.}
\newblock The initial-boundary value problem for the biharmonic
  {S}chr\"{o}dinger equation on the half-line.
\newblock {\em Commun. Pure Appl. Anal. 18}, 6 (2019), 3285--3316.

\bibitem{SteinHA}
{\sc Stein, E.~M.}
\newblock {\em Harmonic Analysis: real-variable methods, orthogonality, and
  oscillatory integrals}, vol.~43 of {\em Princeton Mathematical Series}.
\newblock Princeton University Press, 1993.

\bibitem{Str11}
{\sc Strauss, W., and Bu, C.}
\newblock An inhomogeneous boundary value problem for nonlinear
  {S}chr\"{o}dinger equations.
\newblock {\em J. Differential Equations 173}, 1 (2001), 79--91.

\bibitem{Strich}
{\sc Strichartz, R.~S.}
\newblock Restrictions of {F}ourier transforms to quadratic surfaces and decay
  of solutions of wave equations.
\newblock {\em Duke Math. J. 44}, 3 (1977), 705--714.

\end{thebibliography}
\end{document}